\documentclass[a4paper,11pt]{article}
\usepackage[utf8]{inputenc}
\usepackage{amsfonts}
\usepackage{amsthm}
\usepackage{amsmath}
\usepackage{amsfonts}
\usepackage{latexsym}
\usepackage{amssymb}
\usepackage{dsfont}
\usepackage{graphicx}
\usepackage[usenames]{color}
\usepackage{algorithmic}
\usepackage[ruled,section]{algorithm}

\newtheorem{fed}{Definition}[section]
\newtheorem*{fed*}{Definition}
\newtheorem*{feds*}{Definitions}
\newtheorem{teo}[fed]{Theorem}
\newtheorem*{teo*}{Theorem}
\newtheorem{lem}[fed]{Lemma}
\newtheorem{cor}[fed]{Corollary}
\newtheorem{pro}[fed]{Proposition}
\theoremstyle{definition}
\newtheorem{rem}[fed]{Remark}

\newtheorem*{rems*}{Remarks}

\newtheorem{exa}[fed]{Example}

\newtheorem{nota}[fed]{Notation}
\newtheorem{prob}[fed]{Problem}

\def\coma{\, , \, }

\def\da{^\downarrow}

\def\py{\peso{and}}
\newcommand{\peso}[1]{ \quad \text{ #1 } \quad }

\def\n0{n_{ \text{\rm \tiny o}}}

\def\bce{\begin{center}}
\def\ece{\end{center}}

\def\cD{\mathcal D}

\def\py{\peso{and}}
\def\rk{\text{\rm rk}}
\def\noi{\noindent}

\def\QED{\hfill $\square$}
\def\EOE{\hfill $\triangle$}

\def\bm{\left[\begin{array}}
\def\em{\end{array}\right]}
\def\ben{\begin{enumerate}}
\def\een{\end{enumerate}}
\def\bit{\begin{itemize}}
\def\eit{\end{itemize}}
\def\barr{\begin{array}}
\def\earr{\end{array}}
\def\igdef{\ \stackrel{\mbox{\tiny{def}}}{=}\ }

\def\la{\lambda}

\def\N{\mathbb{N}}
\def\R{\mathbb{R}}
\def\C{\mathbb{C}}
\def\I{\mathbb{I}}

\def\cA{\mathcal{A}}

\def\cC{\mathcal{C}}

\def\cE{\mathcal{E}}

\def\cH{\mathcal{H}}
\def\cK{\mathcal{K}}

\def\cR{{\cal R}}
\def\cS{{\cal S}}
\def\cT{{\cal T}}
\def\cM{{\cal M}}

\def\cN{{\cal N}}
\def\cV{{\cal V}}

\def\cW{{\cal W}}

\def\cX{\mathcal{X}}
\def\cY{\mathcal{Y}}

\def\inc{\subseteq}

\newcommand{\dia}[1] {\text{diag}(#1)}

\def\beq{\begin{equation}}
\def\eeq{\end{equation}}

\def\pausa{\medskip\noi}

\usepackage{graphicx} 
\usepackage{wrapfig} 
\usepackage{caption} 
\usepackage{subcaption}

\begin{document}

\title{Block subspace expansions for eigenvalues and eigenvectors approximation}
\author{
Francisco Arrieta Zuccalli, Pedro Massey and Demetrio Stojanoff 
\footnote{Partially supported by CONICET
(PICT ANPCyT 1505/15) and  Universidad Nacional de La Plata (UNLP 11X974) 
e-mail addresses:  farrieta@mate.unlp.edu.ar
, massey@mate.unlp.edu.ar , demetrio@mate.unlp.edu.ar  }
\\
{\small Centro de Matem\'atica, FCE-UNLP,  La Plata
and IAM-CONICET, Argentina }
}
\date{}

\maketitle

\begin{abstract}
Let $A\in\C^{n\times n}$ and let $\cX\subset \C^n$ be an $A$-invariant subspace with $\dim \cX=d\geq 1$,
corresponding to exterior eigenvalues of $A$. Given an initial subspace $\cV\subset \C^n$ with $\dim \cV=r\geq d$, we search for expansions of $\cV$ of the form $\cV+A(\cW_0)$, where $\cW_0\subset \cV$ is such that $\dim \cW_0\leq d$ and such that the expanded subspace is closer to $\cX$ than the initial $\cV$. We show that there exist (theoretical) optimal choices of such $\cW_0$, in the sense that $\theta_i(\cX,\cV+A(\cW_0))\leq \theta_i(\cV+A(\cW))$ for every $\cW\subset \cV$ with $\dim \cW\leq d$, where $\theta_i(\cX,\cT)$ denotes the $i$-th principal angle between $\cX$ and $\cT$, for $1\leq i\leq d\leq \dim \cT$. We relate these optimal expansions to block Krylov subspaces generated by $A$ and $\cV$.
We also show that the corresponding iterative sequence of subspaces constructed in this way approximate $\cX$ arbitrarily well, when $A$ is Hermitian and $\cX$ is simple. We further introduce computable versions of this construction and compute several numerical examples that show the performance of the computable algorithms and test 
our convergence analysis. 

\end{abstract}

\noindent  MSC(2000) subject classification: 65F15, 15A18, 65F10.



\section{Introduction}

The approximation of eigenvalues and eigenvectors of a large complex matrix $A\in\C^{n\times n}$ is a
central topic in numerical linear algebra.
Given an initial search subspace $\cV_0\subset \C^n$, there are several projection methods that allow
to compute approximations of the eigenvalues and eigenvectors of $A$ based on $\cV_0$. This has motivated the development
of different techniques to expand the search subspace 
$\cV_0\subset \cV_1(A,\cV_0)$ in such a way that the enlarged subspace provides better approximations of eigenvalues and eigenvectors of $A$ than the initial $\cV_0$. One popular expansion method is the block Krylov subspace method (for a detailed account on these methods see \cite{Saad,Stew01,vdH02}). In this work we consider (optimal) expansions of the search subspace based on augmentation/deflation techniques that extend the recent results in \cite{Jia22} for the approximation eigenvalues and eigenvectors to the case of simultaneous approximation of the exterior $d$ eigenvalues and eigenvectors of $A$ as follows.

\pausa
Consider a $d$-dimensional subspace $\cX\subset \C^n$ that is invariant for $A\in \C^{n\times n}$. In this case the eigenvalues $\la_1,\ldots,\la_d$ of the restriction of $A$ to $\cX$ are (some of the) eigenvalues of $A$; we further assume that these 
are exterior eigenvalues. 
For example, if $A$ is Hermitian 
we can be interested in computing the smallest eigenvalues of $A$ and their corresponding eigenvectors. On the other hand, if $A=(B^*B)^{1/2}$ is the operator modulus of an arbitrary matrix $B\in \C^{m\times n}$ then we can be interested in computing the largest 
eigenvalues of $A$ and their corresponding eigenvectors, that are the largest singular values and their corresponding right-singular vectors of $B$ (that generate the so-called dominant subspaces of $B$). In these situations, the information contained in $(\cX,\{\la_i\}_{i=1}^d)$ can be used to approximate the action of $A$ as a linear operator (for example, if $A=(B^*B)^{1/2}$ then we can use $(\cX,\{\la_i\}_{i=1}^d)$ to construct optimal low rank approximations of $B$). When $A$ is a large matrix the direct computation
of $(\cX,\{\la_i\}_{i=1}^d)$ can require too much time, or memory space, or it could simply be impractical. Thus, one of the major tasks in numerical linear algebra is to compute efficient and fast approximations of this data.

\pausa
In this setting we are interested in computing approximations  of $\cX$ in terms of iterative algorithms of the following type: given a subspace $\cV\subset \C^n$ such that $d:=\dim \cX\leq r:=\dim \cV$, 
the algorithm expands (in each step)
the search space $\cV$ by considering $\cV+A(\cW_0)$, for some subspace $\cW_0\subset \cV$ with
$\dim \cW_0=d$. We are interested in optimal choices of the expanded subspaces $\cV+A(\cW_0)$ (in each step). Hence, we further impose that $\cV+A(\cW_0)$ lies the closest to $\cX$ among 
subspaces of the form $\cV+A(\cW)$, for $\cW\subset \cV$ with $\dim \cW=d$. 
As a (vector valued) measure of distance between subspaces we consider the so-called principal
angles. Thus, we search for $\cW_0\subset \cV$ with
$\dim \cW_0=d$ and such that 
\begin{equation}\label{eq intro 1}
\Theta(\cX\coma \cV+A(\cW_0))\leq \Theta(\cX\coma \cV+A(\cW)) 
\end{equation}
for every subspace $\cW\subset \cV$ with $\dim \cW=d$. Here
$\Theta(\cX,\cY)$ denotes the diagonal matrix with diagonal entries
given by the principal angles between $\cX$ and $\cY$, denoted $0\leq \theta_1(\cX,\cY)\leq \ldots\leq \theta_k(\cX,\cY)\leq \pi/2$,
where $k=\min\{\dim \cX,\dim\cY\}$. In case $d=1$ i.e., $\cX=\text{span}\{x_1\}$, this problem was originally considered by
Ye \cite{Ye08}. Recently, Jia \cite{Jia22} has extended the results in \cite{Ye08} and 
obtained several fundamental insights related with computable iterative algorithms that provide approximations of the optimal subspace expansion $\cV+A(\text{span}\{w_0\})$. As mentioned in \cite{Jia22,Ye08} even in this particular setting, the choice of $w\in\cV$ has a relevant impact on the proximity between $\cV+\text{span}\{A(w)\}$ and the target subspace $\cX$, when the initial subspace $\cV$ lacks of
structure (which is the typical case for initial subspaces $\cV$ induced by random matrices).

\pausa
In this work we show that 
there exist subspaces $\cW_0$ as in Eq. \eqref{eq intro 1}. 
 Our approach differs from those
in \cite{Jia22,Ye08}. Indeed, we replace the optimization arguments from those works (involved in the computation of principal angles) with 
a rather simple optimization problem: namely, for subspaces $\cS,\, \cX\subset \C^n$ with $\dim\cS\geq \dim \cX=d$ our approach is based on the computation of those subspaces $\cT\subset \cS$ such that $\dim \cT=d$ and $\Theta(\cX,\cS)=\Theta(\cX,\cT)$. As for the case
$d=1$ (see \cite{Jia22,Ye08}), the construction of such $\cW_0$ involves the target subspace 
$\cX$. Thus, although the expansions $\cV+A(\cW_0)$ are optimal, they are not 
computable and hence have no use for the numerical approximation of $\cX$.
Nevertheless, the {\it theoretical} iterative algorithm derived from 
the optimal (block) subspace expansions $\cV+A(\cW_0)$ serves as a reference for the
assessment of the numerical performance of computable iterative algorithms that 
in each step expand the initial subspace $\cV$ by considering 
$\cV+A(\tilde \cW)$, for some computable $\tilde \cW\subset \cV$ with $\dim \tilde \cW\leq d$.
As a first step towards the analysis of this general theoretical algorithm, we consider
an Hermitian matrix $A\in \C^{n\times n}$ and obtain a proximity analysis of 
the (finite) sequence $(\cV_t)_{t=0}^q$ 
of subspaces such that $$\cV_0=\cV\, ,\ \cV_t=\cV_{t-1}\oplus\cN_t\subset \cV_{t-1}+A(\cV_{t-1})
\peso{and} \dim\cN_t\leq d$$
in such a way that $\cN_t$ has an explicit formula in terms of $A$, $\cV_{t-1}$ and $\cX$, and
$$ \Theta(\cX,\cV_t)=\Theta(\cX,\cV_{t-1}+A(\cV_{t-1})) \,, \ \ t\in \I_q\,$$ 
where $\I_q=\{1,\ldots,q\}$. We remark that $\cV_t=\cV_{t-1}+A(\cW_t)$, for some $\cW_t\subset \cV_{t-1}$ with $\dim \cW_t\leq d$, for $t\in\I_q$ (with the notation of the first part of this paragraph, we relax the condition $\dim \cW_0=d$ and consider optimal expansions with $\dim \cW_0\leq d$). In this context we show that if $\|\cdot\|$ is an arbitrary unitarily invariant norm then $\|\tan(\Theta(\cX,\cV_t))\|$ decays as $((\la_{d+1}-\la_n)/(\la_d-\la_n))^t$, where $\la_1\geq \ldots\geq \la_n$ denote the eigenvalues of $A$ counting multiplicities and arranged in non-increasing order.  Our novel perspective on the optimal (block) subspace expansion problem allow us to show that the (computable) block Krylov sequence of subspaces $$\cK_0=\cV \,, \ \cK_t=\cK_{t-1}+A(\cK_{t-1}) \peso{ for } t\in \I_q\ ,$$ is a natural frame (and a source for comparison) for the (theoretical) sequence $\{\cV_t\}_{t=0}^q$. 

\pausa
Once the proximity analysis of the optimal (block) subspace expansion algorithm is obtained in the Hermitian case, similar to \cite{Jia22} we develop some computable versions of the algorithm based on
 the Rayleigh-Ritz or the refined  Rayleigh-Ritz (projection) methods \cite{J94,J97,J98,Jia04,Jia22,SJia01}. 
Thus, we consider derived computable iterative algorithms that produce finite sequences 
$\{\widehat \cV_t\}_{t=0}^q$ with 
$$\widehat \cV_0=\cV\, ,\ \widehat \cV_t=\widehat \cV_{t-1}\oplus\widehat \cN_t \ , \ \ t\in \I_q$$
such that in each step  they expand the search space by 
$\widehat \cN_t$, which is a computable version of the (theoretical) subspace $\cN_t$ considered above.
Since our theoretical approach to the optimal expansion problem differs from previous works, so  do our derived numerical algorithms. We include implementations of these computable algorithms, together with some numerical examples that allow us to compare the different algorithms: block Krylov method, optimal (theoretical) expansions and computable expansions of the search spaces for the approximation of invariant subspaces associated with exterior (largest) eigenvalues of Hermitian matrices. Even when $\cV_t \subset \cK_t$, 
$\dim \cK_t=(r+1)\cdot t$ and $\dim \cV_t=r+d\cdot t$ (in the generic case), for $t\in\I_q$, 
the examples show that in some cases the subspaces $\cV_t$ (and their computable counterparts) provide approximations of $\cX$ that are comparable with those obtained from $\cK_t$. In these cases, the $\cV_t$'s provide good approximations of $\cX$ through subspaces of much lower dimension than the dimensions of the corresponding $\cK_t$'s.
On the other hand, there are numerical examples in which the block Krylov method outperforms 
the theoretical approximations provided by the $\cV_t$'s and (their computable counterparts).

\pausa
 Our work can be framed into 
the design and analysis of (block) iterative search subspace expansion algorithms for the simultaneous approximation of eigenvalues and eigenvectors of large matrices. Notice that the expansion is obtained in terms 
of an augmentation-deflation technique namely $\cV\mapsto \cV+A(\cV)\mapsto \cV+A(\cW_0)$ where $\cV+A(\cW_0)\subset \cV+A(\cV)$, or $\cV\mapsto \cV+A(\cV)\mapsto \cV\oplus\cN$ where $\cV\oplus\cN\subset \cV+A(\cV)$. Hence, our work can be considered within (one-step) restarted block Krylov algorithms for the approximation of eigenvectors and eigenvalues of large matrices.
 There are several works that deal with similar problems. As we have already mentioned, in \cite{Jia22,Ye08} Ye and Jia considered the optimal (theoretical) expansion algorithms together with some of their computable versions, in which in each step the search space increases its dimension by one. In \cite{LeeS07} (see also \cite{Lee07}) Lee and Stewart 
considered a block expansion of the search subspace in terms of residual vectors corresponding to the Rayleigh-Ritz method. In \cite{WZ14}, this approach was modified by considering residual vectors corresponding to the refined Rayleigh-Ritz method, introduced by Jia in \cite{J97}. Another popular iterative method for the approximation of interior eigenvalues and eigenvectors based on the expansion of the search subspace is the Jacobi-Davidson method, introduced in \cite{Slvd96,Slvd00} (see also \cite{Hochs01}). All these methods are constructed in such a way that the expanded search space provides better approximations of eigenvalues and eigenvectors of matrices. Most of these methods can be implemented using specific (mathematical) software. While our approach is based on a theoretical (i.e. not computable) algorithm, we show that its numerical computable versions (induced by simple projection methods) also provide 
good numerical approximations with the advantage of controlling the size of the stored information in each step (through the control of the dimension of the expanded search space). Furthermore, our computable approach can be extended to deal with the approximation of interior eigenvalues and eigenvectors, by applying different projections methods (i.e. harmonic Rayleigh-Ritz,  refined harmonic Rayleigh-Ritz and shift-invert  methods, see \cite{Jia05,JiaLi14,JiaLi15}).

\pausa
The paper is organized as follows. We include some brief preliminaries in Section \ref{sec prelis}. 
In Section \ref{sec main results general} we present our main results and delay their proofs until Section \ref{sec proofs main}. In more detail, in Section \ref{sec main1} we 
obtain a complete solution of the optimal subspace expansion problem. We further
propose a variation of the optimal expansion problem and describe an explicit algorithmic construction of solutions to this problem (in terms of closed formulas involving the target subspace). We show that the proposed algorithm constructs a sequence of increasing subspaces 
that are related to some block Krylov subspaces.
In Section \ref{sec main2} we obtain a proximity 
analysis of the algorithmic construction of the optimal subspace expansions 
together with a proximity analysis of associated block Krylov subspaces in the Hermitian case. In Section 
\ref{sec compu subop} we follow recent ideas of Jia \cite{Jia22} and consider algorithms that 
construct computable expansions of subspaces based on projection methods.
In Section \ref{sec proofs main} we present the proofs of our main results, based on some technical results.
We defer the proof of these technical results to Section \ref{sec append complet}.
Finally, in Section \ref{sec num exa} we include the analysis of several numerical examples 
obtained by implementation the Algorithms considered in Sections \ref{sec main2} and \ref{sec compu subop}. 

\section{Preliminaries} \label{sec prelis}

\pausa
{\bf Notation and terminology}. We let $\C^{m\times n}$ be the space of complex $m\times n$ matrices with entries in $\C$. A norm $\|\cdot \|$ in $\C^{n\times n}$ is unitarily invariant (briefly u.i.n.) if 
$\|UAV\|=\|A\|$, for every $A\in \C^{n\times n}$ and unitary matrices $U,\, V\in \C^{n\times n}$.
For example, the operator and Frobenius norms, denoted by $\|\cdot\|_2$ and $\|\cdot \|_F$ respectively, are u.i.n.'s.  If $\cV\subset \C^n$ is a subspace, we let 
$P_\cV\in\C^{n\times n}$ denote the orthogonal projection onto $\cV$.

\pausa
For $A\in\C^{m\times n}$ we will denote it's Moore–Penrose pseudo-inverse as $A^\dagger$ . Among other basic properties of the pseudo-inverse we will use the fact that $AA^\dagger=P_{R(A)}$ and $A^\dagger A=P_{\ker(A)^\perp}=P_{R(A^*)}$ where $R(A)$ denotes the subspace of $\C^m$ spanned by the columns of $A$.

\pausa
For $n\in\N$, let $\I_n=\{1,\ldots,n\}$. 
Given a vector $x\in\C^n$ we denote by $\text{diag}(x)\in \C^{n\times n}$ the diagonal matrix whose main diagonal is $x$.
If  $x=(x_i)_{i\in\I_n}\in\R^n$ we denote by $x\da=(x_i\da)_{i\in\I_n}$ the vector obtained by 
rearranging the entries of $x$ in non-increasing order. We also use the notation
$(\R^n)\da=\{x\in\R^n\ :\ x=x\da \}$ and $(\R_{\geq 0}^n)\da=
\{x\in\R_{\geq 0}^n\ :\ x=x\da \}$. 

\pausa
 Given an Hermitian matrix $A\in \C^{n\times n}$ 
we denote by $\la(A)=(\la_i(A))_{i\in\I_n}\in (\R^n)\da$ 
the eigenvalues of $A$ counting multiplicities and arranged in 
non-increasing order.   
For $B\in \C^{m\times n}$ we let $s(B)=\la(|B|)\in (\R_{\geq 0}^n)\da$ denote the singular values of $B$, i.e. the eigenvalues of the positive semi-definite matrix $|B|=(B^*B)^{1/2}\in\C^{n\times n}$. 

\pausa
{\bf Principal angles between subspaces}. 
Let $\cS,\,\cX\subset \C^n$ be subspaces such that $d=\dim \cX\leq \dim \cS$.
Recall that the principal angles between $\cX$ and $\cS$, denoted $\theta_1,\ldots,\theta_d\in [0,\pi/2]$
are given by $$\cos(\theta_i)=s_i(P_\cS P_\cX)=s_i(P_\cX P_\cS) \peso{for} 1\leq i\leq d\,.$$ Notice that by construction 
$0\leq \theta_1\leq \ldots\leq \theta_d\leq \pi/2$.
We let 
$$
\Theta(\cX,\,\cS)= \dia{\theta_1,\ldots,\theta_d}\, .
$$ Principal angles allow to define metrics between 
subspaces \cite{CWL}, which are use in a wide range of applications (see \cite[Section 1]{CWL} and the references therein). 
In order to save notation, whenever we have matrices $X$ and $V$ with the same number of rows we will write $\Theta(X,V)$ instead of $\Theta(R(X),R(V))$. \EOE

\section{Main results}\label{sec main results general}

In this section we present our main results. In Section \ref{sec main1} we solve the optimal (block) subspace expansion problem (for details see below) and find an alternative (relaxed) formulation for an optimal expanded search space. This alternative formulation induces an iterative process that we describe in terms of a pseudo-code in Algorithm \ref{algo1}. We compare the sequence of subspaces that are the outputs of Algorithm \ref{algo1} with the sequence of block Krylov spaces, where both sequences are induced by a matrix $A\in \C^{n\times n}$ and an initial subspace $\cV\subset \C^n$. In Section \ref{sec main2} we obtain a proximity analysis of the output of Algorithm \ref{algo1} and of the sequence of block Krylov spaces, to the eigenspaces of a Hermitian matrix $A$ associated to its exterior (largest) eigenvalues. In Section \ref{sec compu subop} we follow ideas from \cite{Jia22} and describe pseudo-codes in Algorithm \ref{algo2} for the iterative construction of computable (sub)optimal (block) subspace expansions. In Section \ref{sec num exa} we include several numerical examples related to the algorithms considered in Sections \ref{sec main2} and \ref{sec compu subop}.

\subsection{Optimal (block) subspace expansion problem}\label{sec main1}

Our first results are related to the optimal (block) subspace expansion problem.
Hence, we fix a $d$-dimensional subspace $\cX\subset \C^n$ that is invariant for $A\in \C^{n\times n}$. 
We consider an initial (guess) subspace $\cV\subset\C^n$ such that $r=\dim\cV \geq d$. In this setting, 
we search for the optimal subspaces $\cW_0\subset \cV$ with $\dim \cW_0=d$ and such that 
\begin{equation}\label{eq main1}
\Theta(\cX\coma \cV+A(\cW_0))\leq \Theta(\cX\coma \cV+A(\cW)) 
\end{equation} for every subspace $\cW\subset \cV$ with $\dim \cW=d$. 
If $\cW\subset \cV$ with $\dim \cW=d$, let $\cE= \cV+A(\cW)$ and notice that 
$$\cE\subset \cS:=\cV+A(\cV) \peso{is such that} \cV\subset \cE \py \dim \cE\leq r+d\,. $$
Thus, as a first step towards the solution of \eqref{eq main1}, we can consider the following relaxation of the
problem in Eq. \eqref{eq main1}: find $\cE_0\subset \cS$ such that $\cV\subset \cE_0$, $\dim \cE_0\leq r+d$ and such that 
\begin{equation}\label{eq main2}
\Theta(\cX\coma \cE_0)\leq \Theta(\cX\coma \cE) 
\end{equation} for every such $\cE$ as above. The advantage of considering the problem in Eq. \eqref{eq main2} is 
that its solution can be described in simple terms: indeed, if we let $\cM=P_\cS(\cX)\subset \cS$ then, every
$\cV+\cM\subset \cE_0\subset \cS$ with $\dim \cE_0\leq r+d$ satisfies Eq. \eqref{eq main2}. Furthermore, using the fundamental identity 
\begin{equation}\label{eq main3}
\cS=\cV+A(\cV)= \cV\oplus (1-P_\cV) A(\cV)
\end{equation}
we can easily construct $\cE_0$ as above such that $\cE_0=\cV+A(\cW_0)$, for some $\cW_0\subset \cV$ with
$\dim \cW_0=d$ (for the proofs of both of these assertions, see Section \ref{sec proofs main}). Notice that, since $\cE_0$ is a solution to the relaxed problem, such $\cW_0$ satisfies Eq. \eqref{eq main1} and hence, it is a solution to the original optimal subspace expansion problem. To describe the main result of this section we include the following: 
\begin{nota}\label{notac1} In what follows we consider:
\begin{enumerate}
\item $A\in \C^{n\times n}$  and an $A$-invariant subspace $\cX\inc \C^n$, with $\dim \cX=d$.
\item A subspace $\cV\subset \C^n$  with $\dim \cV :=r\geq d$, $\cX\not\subseteq\cV$ and 
$A(\cV)\not\subseteq \cV$.
\item Let $V\in \C^{n\times r}$ have orthonormal columns such that $R(V)= \cV$ and set
$$
R:=AV-V(V^*AV)=(1-P_\cV)AV\in\C^{n\times r}\,.
$$ \EOE
\end{enumerate}
\end{nota}

\begin{teo}\label{teo opti 1} Consider Notation \ref{notac1}.
Denote by $\cS := \cV + A(\cV)$. Then,
\begin{enumerate}
\item[1.] There exist  
subspaces $\cW_0 \inc \cV$ with $\dim \cW_0 = d$ such that
\beq\label{los W0}
\Theta(\cX\coma\cV+A(\cW_{0})) =\Theta(\cX\coma\cS) = \min \{ \Theta(\cX\coma\cV+A(\cW)):  
\cW\inc \cV \coma \dim \cW=d\} \ .
\eeq
\end{enumerate}
If we let $\cN := (1-P_\cV)(P_\cS(\cX)\,)$  
 and 
$\cC_{\rm op} = R^\dagger (\, \cN\,) \inc \ker R^\perp \inc\C^r  \,$ then
\begin{enumerate}
\item[2.]  $\cC_{\rm op}=R^\dagger \cX$ , $\cN=RR^\dagger \cX=R(\cC_{\rm op})$, and $\cS=\cV \oplus R(\C^r)$.
\item[3.]
The subspaces $\cW_0$ which satisfy Eq. \eqref{los W0} are exactly those of the form 
\beq\label{el W0}
\cW_0 = V(\cC_0) \peso{for some \ \ $\cC_0\inc \C^r $ \ \ 
with  \ \ $\dim \cC_0 = d$ \ \ and} \cC_{\rm op} \inc  P_{\ker R^\perp} \cC_0 \ .
\eeq
In particular, $\cV\oplus \cN=\cV\oplus R(\cC_{\rm op})\subset \cV\oplus R(\cC_0)= \cV\oplus (1-P_\cV)A(\cW_0)= \cV+A(\cW_0)$.
\item[4.] If $\cV\cap A^{-1}(\cV)=\{0\}$ and $\dim \cN 
= d$, 
then $\cW_0$ as in Eq. \eqref{los W0} is unique and given by 
$$ \cW_0 =  V (\cC_{\rm op}) = V \, R^\dagger (\cX) \peso{with} \cV+A(\cW_0)=\cV\oplus \cN\,.$$
\end{enumerate}
\end{teo}
\proof See Section \ref{sec proofs main}. \QED

\begin{rem} We show how to deduce \cite[Theorem 2.2]{Jia22} from our previous result. Indeed,
consider the notation from Theorem \ref{teo opti 1} and assume that $d=1$. 
Hence, $\cX=\C\cdot x$, for some $x\in\C^n$ such that $Ax=\la\,x$.
Assume further (as in \cite[Theorem 2.2]{Jia22}) that $x\notin \cV$, $\cV\cap A^{-1}(\cV)=\{0\}$ and that $RR^\dagger x\neq 0$ (i.e. $\dim\cN=1=d$). 

\pausa
By item 4. in Theorem \ref{teo opti 1} we get that the unique optimal subspace $\cW_0$ is spanned by 
$w_0=VR^\dagger x$ and that $\cV_{\rm op}:=\cV+\C\cdot A(w_0)=\cV\oplus \C \cdot RR^\dagger x$. This last fact implies that $(1-P_\cV)A(w_0)\in\C\cdot RR^\dagger x$. By items 1. and 2. in Theorem \ref{teo opti 1} we get that $$\angle(\cV_{\rm op},x)=\angle (\cS,x)=\angle (P_{\cS}x,x)=\angle (P_{\cV}x+RR^\dagger x,x)$$
Moreover, our approach shows that the auxiliary subspace $\cV_R=\cV\oplus R(\C^r)$ considered in \cite{Jia22} coincides with $\cV_R=\cS \,(=\cV+A(\cV))$ (see item 2. in Theorem \ref{teo opti 1}). Indeed, we have exploited this last fact by introducing (in the general case $d\geq 1$) a natural comparison between the subspaces $\cV_{\rm op}$ and the block Krylov subspaces $\cS=\cV+A(\cV)$ throughout our work (see Remark \ref{rem sobre algo1} below). \EOE
\end{rem}

\pausa
Consider the notation from Theorem \ref{teo opti 1}. As pointed out by Jia in \cite{Jia22} in case $d=1$, it is convenient to focus on the optimal expanded space $\cV+A(\cW_0)\subset \C^n$ rather than the subspace $\cW_0\subset \cV$. After all, we are interested in the optimal search space $\cV+A(\cW_0)$
which contains $d$-dimensional subspaces that are the closest to $\cX$, among all such subspaces. Furthermore, if we are willing to 
relax the problem and consider optimal search subspaces $\cV+A(\cW')$ for subspaces $\cW'\subset \cV$ with 
$\dim \cW'\leq d$ (instead of asking that $\dim\cW'=d$), then we can obtain simple procedures to compute such optimal subspaces. The following result, that describes optimal search subspaces $\cV+A(\cW')$ as above, will play a key role in the rest of our work. 

\begin{cor}\label{coro para aplic}
Consider Notation \ref{notac1}. Let $\cS=\cV+A(\cV)$ and $\cN = (1-P_\cV)(P_\cS(\cX)\,)$ as in Theorem \ref{teo opti 1}. Then, 
\begin{enumerate}
\item $\cV\oplus \cN=\cV+P_\cS(\cX)\inc \cS$, $\dim \cN\leq d$ so $\dim \cV\oplus \cN\leq r+d$;
\item If we let $\cW '=VR^\dagger (\cN) \subset \cV$ then $\dim\cW'\leq d$ and  $$\cN=(1-P_\cV)A(\cW')
\implies \cV\oplus \cN=\cV+A(\cW')\,.$$ 
\item 
In the general case, $\Theta(\cX\coma\cV\oplus\cN) = \Theta(\cX\coma\cS)$. Hence, $$ \Theta(\cX\coma\cV\oplus\cN)=\Theta(\cX\coma\cV+A(\cW'))= \min \{ \Theta(\cX\coma\cV+A(\cW)):  
\cW \subset \cV \coma \dim \cW\leq d\} \ .$$
\item \label{cor it4}If we assume further that $\dim \cN=d$ then, for every $\cW_0\subset \cV$ satisfying Eq. \eqref{los W0} we have that $$\cV+ A(\cW_0)=\cV\oplus\cN\,.$$
\end{enumerate}
\end{cor}
\proof See Section \ref{sec proofs main}. \QED

\pausa
Notice that Corollary \ref{coro para aplic} provides a simple construction (with explicit formulas in terms of $A$, $\cX$ and $\cV$) of an optimal expanded search subspace $\cV\oplus \cN$  
(compare item 1. in Theorem \ref{teo opti 1} with
item 3. in Corollary \ref{coro para aplic}) with the additional property that $\cV$ and $\cN$ are mutually orthogonal. We can iterate this construction and obtain an algorithmic procedure that constructs sequences of step-wise optimal expanded search spaces as follows.

\begin{algorithm}
\caption{Optimal theoretical (block) subspace expansion}\label{algo1}
\centerline{
}
\begin{algorithmic}[1]
\REQUIRE $A\in\C^{n\times n}$, $\cX,\,\cV\subset \C^n$ as in Notation \ref{notac1}; an integer $q\geq 0$ (number of iterations).

\medskip

\STATE Set $\cV_0=\cV$. For $1\leq t\leq q$ define recursively: 
\begin{enumerate}
\item[{\small 1.1:}] $\cS_t=\cV_{t-1}+A(\cV_{t-1}) $;
\item[{\small 1.2:}]  $\cN_t=(1-P_{\cV_{t-1}})P_{\cS_t}(\cX)$; 
\item[{\small 1.3:}] $\cV_t=\cV_{t-1}\oplus \cN_t$. 
\end{enumerate}
\STATE{\bf Return:} Subspaces $\cV=\cV_0\subset \ldots\subset \cV_q\subset \C^n$, such that 
$\Theta(\cX,\cS_t)=\Theta(\cX,\cV_t)$, for $t\in\I_q$.
\end{algorithmic}
\end{algorithm}

\begin{rem}[Algorithm \ref{algo1} vs. block Krylov subspace method]\label{rem sobre algo1}
Consider Notation \ref{notac1} and fix $q\geq 0$. Let $\cV_0\subset \ldots\subset \cV_q\subset \C^n$ constructed by Algorithm \ref{algo1}. Notice that given $\cV_{t-1}$ then $\cV_t$ is the optimal subspace expansion of $\cV_{t-1}$ considered in Corollary \ref{coro para aplic}, for $1\leq t\leq q$.
By construction, $\cV_0=\cV$ and for $1\leq t\leq q$ we get that
$$
\cV_t=\cV_{t-1}\oplus \cN_t =\cV_{t-1}+P_{\cS_t}(\cX)\subset \cV_{t-1}+A(\cV_{t-1})=\cS_t\,.
$$
Consider the increasing family of block Krylov subspaces $\cK_0\subset \ldots\subset \cK_q\subset \C^n$ constructed from $A$ and $\cV$. Recall that $\cK_0=\cV$ and 
for $1\leq t\leq q$, 
$$\cK_t=A^t (\cV) +  A^{t-1}(\cV) + \cdots + A(\cV) + \cV\,.$$  Alternatively, we have that $\cK_t=\cK_{t-1}+A(\cK_{t-1})$, for $1\leq t\leq q$.
It follows that $\cV_t\subset \cK_t$, for $0\leq t\leq q$. Indeed, $\cV_0=\cK_0$ and if we assume that for some $1\leq t\leq q$ we have that $\cV_{t-1}\subset \cK_{t-1}$, then 
$$
\cV_t\subset \cV_{t-1}+A(\cV_{t-1})\subset \cK_{t-1}+A(\cK_{t-1})=\cK_t\,.
$$ 
Moreover, by construction $\dim \cV_t\leq r+d\cdot t$ while $\dim \cK_t\leq r\cdot (t+1)$, for $0\leq t\leq q$. Hence, we can expect the subspaces $\cV_t$ to be of a (much) lower dimension than the corresponding $\cK_t$, for $t\geq 1$. We point out that in the numerical examples considered in Section \ref{sec num exa} we observed that $\dim \cN_t=d$ and hence $\dim \cV_t= r+d\cdot t$, while $\dim \cK_t= r\cdot (t+1)$, for $0\leq t\leq q$. In particular, in each step the Algorithm \ref{algo1} constructed the optimal 
subspace expansion for $A$ and the corresponding $\cV_t$ (see item \ref{cor it4} in  Corollary \ref{coro para aplic}).
\EOE
\end{rem}

\pausa
Consider the notation in Remark \ref{rem sobre algo1}. There is a fundamental difference
 between the finite sequences $(\cV_t)_{t=0}^q$ and $(\cK_t)_{t=0}^q$ namely, that 
the latter is a computable sequence of subspaces that do not depend on the target subspace $\cX$. 
On the other hand, the family $(\cV_t)_{t=0}^q$ can not be computed unless the target 
(typically unknown) subspace $\cX$ is known; hence, from the numerical point of view, $(\cV_t)_{t=0}^q$
does not provide a computable family of approximating subspaces for $\cX$.
Nevertheless, from a theoretical point of view, it is natural to ask under which conditions 
the family $(\cV_t)_{t=0}^q$ approximates $\cX$.

\begin{prob}\label{prob1}
Consider Notation \ref{notac1}. Let $(\cV_t)_{t=0}^q$ be obtained from Algorithm \ref{algo1}. Find conditions (on $A$, $\cX$ and $\cV$) under which we can obtain explicit upper bounds for the decay of 
the principal angles $\Theta(\cX,\cV_t)$, for $t\geq 1$. 
\EOE 
\end{prob}

\pausa
Problem \ref{prob1} in its full generality seems to be hard. Notice that even when $A$ is diagonalizable, (the spectral) representations of $A$ would involve decompositions of the identity 
$Q_1+\ldots+Q_r=I$ in terms of (possibly non-orthogonal) projections such that $Q_iQ_j=0$ when $i\neq j$. Analysis of these situations is typically subtle. As a first step towards a better understanding of solutions of Problem \ref{prob1} we restrict attention to the Hermitian case $A=A^*$ and assume further that $\cX=\text{span}\{x_1,\ldots, x_d\}$ is the subspace spanned by the eigenvectors of $A$ corresponding to the $d$ largest eigenvalues of $A$. We focus on the largest eigenvalues of $A$ only since the smallest eigenvalues of $A$ can be thought of as the largest eigenvalues of $-A$, which is also Hermitian.

\subsection{Analysis of Algorithm \ref{algo1}: exterior spectra in the Hermitian case}\label{sec main2}

In this section we obtain an analysis of Algorithm \ref{algo1} in the case that $A\in \C^{n\times n}$
is an Hermitian matrix and that $\cX=\text{span}\{x_1,\ldots, x_d\}$ is the subspace spanned by the
eigenvectors of $A$ corresponding to the $d$ largest eigenvalues of $A$. Indeed, let $\{\cV_t\}_{t=0}^q$
be constructed as in Algorithm \ref{algo1} with an initial (guess) subspace $\cV$ and let $\la(A)=(\la_i)_{i\in\I_d}\in (\R^n)\da$ denote the eigenvalue list of $A$. If 
we assume that $\la_d>\la_{d+1}$ (which is a generic case) then we obtain an upper bound for $\Theta(\cX,\cV_t)$ that becomes arbitrarily small for large enough $t\geq 1$.

\pausa
To describe the main results of this section we include the following 

\begin{nota}\label{notac22} In what follows we consider:
\begin{enumerate}
\item An Hermitian matrix $A\in \C^{n\times n}$ with eigenvalues list $(\la_i)_{i\in\I_n}\in (\R^n)\da$. 
\item An orthonormal basis $\{x_i\}_{i\in\I_n}$ such that $A\,x_i=\la_i\, x_i$, for $i\in\I_n$.
\item The invariant subspace $\cX=\text{span}\{x_1,\ldots,x_d\}\subset \C^n$ and assume that $\la_d>\la_{d+1}$.
\item A subspace $\cV\subset \C^n$  with $\dim\cV :=r\geq d$, $\cX\not\subseteq\cV$ and 
$A(\cV)\not\subseteq \cV$.\EOE
\end{enumerate}
\end{nota}

\begin{teo}\label{teo main conv vt} 
Consider Notation \ref{notac22}
and let $\{\cV_t\}_{t\in \I_q}$ be constructed according to Algorithm \ref{algo1}. If $\cX\cap\cV^\bot=\{0\}$ then for every unitarily invariant norm $\|\cdot\|$ we have
\begin{equation}\label{eq hay conv1}
    \|\tan(\Theta(\cX,\,\cV_t))\|
    \leq
    \left(\frac{\lambda_{d+1}-\lambda_n}{(\lambda_d-\lambda_n)+(\lambda_d-\lambda_{d+1})}\right)^{t}
    \| \tan(\Theta(\cX, \cV_{0}))\|\, 
    \text{,  for } t\geq 1\,.
 \end{equation}
\end{teo}
\proof See Section \ref{sec proofs main}. \QED

\pausa
Notice that the hypothesis $\cX\cap \cV^\perp=\cX\cap \cV_0^\perp=\{0\}$ holds in a generic case. On the other hand, this condition implies that the tangents of the principal angles $\Theta(\cX,\cV_0)$, that appear in the right-hand side of Eq. \eqref{eq hay conv1}, are well defined.

\pausa
As mentioned in Remark \ref{rem sobre algo1}, the block Krylov subspaces $\cK_0=\cV$ and $\cK_t=\cK_{t-1}+A(\cK_{t-1})$ satisfy that $\cV_t\subset \cK_t$. Hence, our result about the proximity of the $\cV_t$'s to $\cX$ (together with elementary properties of principal angles) imply the proximity of the $\cK_t$'s to $\cX$.
 Moreover, below we obtain an upper bound
for $\Theta(\cX,\cK_t)$ that improves the upper bound derived from that for $\Theta(\cX,\cV_t)$ above. In particular, our upper bound for $\Theta(\cX,\cK_t)$ takes into account the so-called oversampling $\rho=r-d$, where $r=\dim \cV$ denotes the dimension of the initial subspace $\cV$.

\begin{teo}\label{teo main Kt} 
Consider Notation \ref{notac22} and assume that 
$\text{span}\{x_1,\ldots,x_r\}\cap \cV^\perp=\{0\}$. 
Let $\{\cK_t\}_{t=0}^q$ be the block Krylov subspaces induced by $\cV\subset \C^n$ and let 
$d\leq p\leq r$. Then, there exists $\cH_p\subset \cV$ with $\dim\cH_p=d+r-p\geq d$, $\cX\cap \cH_p^\perp=\{0\}$ and such that 
for every unitarily invariant norm $\|\cdot\|$:
		\begin{equation}\label{eq hay convkt}
        \|\tan(\Theta(\cX,\,\cK_t))\|
        \leq
4\, \frac{\la_{p+1}-\la_n}{\lambda_d-\la_n} \, 3^{-t\cdot\min\left\{\sqrt{\frac{\la_d-\la_n}{\la_{p+1}-\la_n} \,-\,1} \, ,\,1\right\}} 
        \, \| \tan(\Theta(\cX,\, \cH_p))\| 
				\ , \ \text{ for } \ t\geq 1
        \,.
    \end{equation}
\end{teo}
\proof See Section \ref{sec proofs main}. \QED

\pausa We point out that the hypothesis $\text{span}\{x_1,\ldots,x_r\}\cap \cV^\perp=\{0\}$ holds in a generic case. On the other hand, the condition $\cX\cap \cH_p^\perp=\{0\}$  implies that the tangents of the principal angles $\Theta(\cX,\cH_p)$, that appear in the right-hand side of Eq. \eqref{eq hay convkt}, are well defined for every $d\leq p\leq r$.

\pausa
The main difference between our proximity analysis of the finite sequences $\{\cV_t\}_{t=0}^q$ and 
$\{\cK_t\}_{t=0}^q$ is that the upper bounds in Theorem \ref{teo main Kt} take into account the oversampling parameter $\rho=r-d$, through the choice of the value of the parameter $p$. Below we consider the impact of the choice of $d\leq p\leq r$ in the upper bounds for the block Krylov subspaces.

\begin{rem}\label{rem sobre p}
Consider the notation in Theorem \ref{teo main Kt}. Let $d\leq p_1\leq p_2 \leq r$; an inspection of the (technical results that allow to obtain the) proof of Theorem \ref{teo main Kt} show that  
$\cH_{p_2}\subset \cH_{p_1}\subset \cV$. Since $\dim \cH_{p_i}=d+r-p_i\geq d$, for $i=1,2$, 
we get that $\Theta(\cX,\cH_{p_1})\leq \Theta(\cX,\cH_{p_2})$; hence,
$\| \tan(\Theta(\cX,\, \cH_{p_1}))\|\leq \| \tan(\Theta(\cX,\, \cH_{p_2}))\|$. On the other hand, since $\la_1\geq \ldots\geq \la_n$ then
$$
 \frac{\la_{p_2+1}-\la_n}{\lambda_d-\la_n} \ 3^{-t\cdot\min\left\{\sqrt{\frac{\la_d-\la_n}{\la_{p_2+1}-\la_n} \,-\,1} \, ,\,1\right\}} 
\leq
 \frac{\la_{p_1+1}-\la_n}{\lambda_d-\la_n} \ 3^{-t\cdot\min\left\{\sqrt{\frac{\la_d-\la_n}{\la_{p_1+1}-\la_n} \,-\,1} \, ,\,1\right\}} \,.$$
Thus, for small values of $t\geq 1$ the upper bound in Eq. \eqref{eq hay convkt} for $p=p_1$ provides better estimates than the corresponding upper bound for $p=p_2$. On the other hand, for larger values of $t$ the upper bound in Eq. \eqref{eq hay convkt} for $p=p_2$ provides better estimates than the corresponding upper bound for $p=p_1$. This applies, in particular, for the case $p_1=d$ i.e. for $\cH_d=\cV$. \EOE
\end{rem}

\pausa
As a final comment, we mention that we have implemented Algorithm 
\ref{algo1} and obtained several numerical examples. We have also implemented
the block Krylov method and computed the upper bounds obtained in Theorems \ref{teo main conv vt} and \ref{teo main Kt}. For the analysis of these examples together with some plots of the results see Section \ref{sec num exa}.

\subsection{Computable subspace expansions}\label{sec compu subop}

As we have already pointed out (see the comments after 
Remark \ref{rem sobre algo1}) the optimal subspace expansions $\cV_t$ do not provide 
approximations of the invariant (target) subspace $\cX$ that are useful from a 
numerical perspective. Thus, we follow ideas from \cite{Jia22} and consider
computable (typically suboptimal, see the numerical examples in Section \ref{sec num exa}) subspace expansions induced by projection methods.
Again, we will focus on the case of exterior eigenvalues of Hermitian matrices, but 
this approach can be extended to more general settings, using different projection methods (e.g. those considered in \cite{Jia05,JiaLi14,JiaLi15}).

\pausa
Consider Notation \ref{notac22} and let $\cS:=\cV+A(\cV)$. Notice that (for $t=1$) Step 1.2. in Algorithm \ref{algo1} includes the computation of $P_\cS(\cX)$; since $\cX$ is the target
subspace (and therefore unknown) \cite{Jia22} suggest that we can replace the subspace
$P_\cS(\cX)$ with some computable subspace that plays it's role. 
These computable replacements of $P_\cS(\cX)$ are motivated by some well known 
projection methods from numerical linear algebra. Indeed, motivated by the 
Raleigh-Ritz method, we can consider the subspace $\widehat \cX=\text{span}\{\widehat x_1,\ldots,\widehat x_d\}$, where the vectors $\widehat x_i$ are the eigenvectors of largest eigenvalues of the compression $A_\cS=P_\cS A P_\cS$ of $A$ to the subspace $\cS$; i.e. $A_\cS \,\widehat x_i=\la_i(A_\cS) \, \widehat x_i$, for $i\in\I_d$, where
$\la(A_\cS)=(\la_i(A_\cS))\in (\R^n)\da$ denotes the eigenvalues of $A_\cS$. Similarly, we can consider
the computable replacement of $\cX$ induced by the refined Raleigh-Ritz (see \cite{J97,J98,Jia04}). Thus, it is convenient to consider the following

\begin{fed}[\cite{Jia22}] For a chosen projection method, the approximation of $\cX$ extracted by it
from $A$ and $\cS=\cV+A(\cV)$ is called the computable replacement of $\cX$, and denoted $\widehat \cX$. \EOE
\end{fed}

\pausa
Based on the previous ideas we consider the following

\begin{algorithm}
\caption{Computable (block) subspace expansion}\label{algo2}
\centerline{
}
\begin{algorithmic}[1]
\REQUIRE $A\in\C^{n\times n}$, $\cX,\,\cV\subset \C^n$ as in Notation \ref{notac1}; an integer $q\geq 0$ (number of iterations). A numerical projection method ``Proj-Method''.

\medskip

\STATE Set $\widehat \cV_0=\cV$. For $1\leq t\leq q$ define recursively: 
\begin{enumerate}
\item[{\small 1.1:}] $\widehat \cS_t=\widehat \cV_{t-1}+A(\widehat \cV_{t-1}) $;
\item[{\small 1.2:}] $\widehat \cX_t=$ Proj-Method$(A,\widehat \cS_t)$ the computable replacement of $\cX$ from $A$ and $\widehat \cS_t$.
 \item[{\small 1.3:}]  $\widehat  \cN_t=(1-P_{\widehat \cV_{t-1}})(\widehat \cX_t)$ 
\item[{\small 1.4:}] $\widehat \cV_t=\widehat \cV_{t-1}\oplus \widehat \cN_t$. 
\end{enumerate}
\STATE{\bf Return:} Subspaces $\cV=\widehat \cV_0\subset \ldots\subset \widehat \cV_q\subset \C^n$.
\end{algorithmic}
\end{algorithm}

\pausa
Recall that  our analysis of the theoretical version of Algorithm \ref{algo2} is obtained for
the exterior eigenvalues of Hermitian matrices. Thus, we have implemented Algorithm \ref{algo2} 
for the case in which $A\in\C^{n\times n}$ is Hermitian and $\cX$ is spanned by the eigenvectors of $A$ corresponding to the largest $d$ eigenvalues. Hence, we have considered the 
Raleigh-Ritz and refined Raleigh-Ritz projection methods for the numerical implementation, which are known to work well in these situations. We have compared these numerical versions with the (theoretical) optimal subspace expansion from Algorithm \ref{algo1} (see Section \ref{sec num exa}).

\pausa
Notice that there is not a direct comparison between our (theoretical) Algorithm \ref{algo1}
and its numerical counterparts in Algorithm \ref{algo2}: indeed, if we choose the same initial
(guess) subspace $\cV=\cV_0$ then we can compare the principal angles between the target subspace and the outputs of these algorithms in the first step (that is, for $t=1$). After that, the outputs produced by these algorithms will not be comparable, as the subspaces produced in the first step will differ.
Since the optimality of the construction in Algorithm \ref{algo1} is {\it step-wise}, it could be 
that the numerical counterparts from Algorithm \ref{algo2} outperform the output from Algorithm \ref{algo1} in the long term.
The numerical examples show that, typically (but not always), the output from Algorithm \ref{algo1} provides a subspace that is closer to the target subspace than the output from Algorithm \ref{algo2}, for a fixed number of iterations. 

\begin{rem}[Comparison of Algorithm \ref{algo2} with previous numerical expansions]\label{rem comparacion}
    To compare the performance of our numerical algorithms, we have considered the Residual Arnoldi method (RA) developed in \cite{Lee07, LeeS07}. 
    We have also considered an extension of Jia's algorithm for computable subspace expansion introduced in \cite[Section 3]{Jia22}; briefly, given $\widehat{\cV}_{t-1}$, we take the computable replacement $\widehat{\cR}_{t-1}$ of $\cX$ from $A$ and the range of the residual matrix $R_{t-1}=(1-P_{\widehat{\cV}_{t-1}})A\widehat{V}_{t-1}$, where $\widehat{V}_{t-1}$ has orthonormal columns that span $\widehat{\cV}_{t-1}$. For this computable replacement we have used the Raleigh-Ritz projection method. Finally, we set $\widehat{\cV}_t = \widehat{\cV}_{t-1}\oplus \widehat{\cR}_{t-1}$.
\end{rem}

\section{Proofs of the main results}\label{sec proofs main}

In this section we present the proofs of our  mains results described
in Sections \ref{sec main1} and \ref{sec main2}. The  proofs 
make use of some technical results; we describe these technical
results here and present their proofs in Section \ref{sec append complet} (Appendix).

\pausa
{\bf Proof of the results in Section \ref{sec main1}}. 
We begin with the following technical results that we need in the sequel.

\begin{pro}\label{pro monotonia de angulos}
Let $\cX,\,\cY$ and $\cS$ be subspaces of $\C^n$ 
such that $d=\dim\cX\leq\dim\cY$ and $\cY\inc\cS$. 
\begin{enumerate}
\item $\Theta(\cX,\,\cS) \leq \Theta(\cX,\,\cY)$ that is, $\theta_i(\cX,\,\cS) \leq \theta_i(\cX,\,\cY)$  every $ i\in \I_d\,$.
\item Denote by  $\cM = P_\cS (\cX)\inc \cS $ and $m = \dim \cM \le d$. Then 
\beq \label{da igual}
\theta_i(\cX,\cS) = \theta_i(\cX,\cM)<\pi/2  \peso{if \ $ 1\le i \le m$ \ 
and \  $\theta_i(\cX,\cS) =\frac\pi2$  
\  if}  m< i\le d  \ . 
\eeq
\item 
$\Theta(\cX,\cS)= \Theta(\cX,\cY) \iff  \cM \inc \cY$.
\end{enumerate}
\end{pro}
\proof See Section \ref{Appendixity1}. \QED

\begin{pro}\label{rem desc sum y algo mas}
Let $\cV \coma \cA \inc \C^n$ be subspaces. 
If we denote by $\cS \igdef \cV+ \cA$, then  
\begin{equation}\label{eq desc sum1}
\cS = \cV+ \cA=\cV\oplus (1-P_\cV)(\cA)  \implies P_\cS=P_\cV+P_{(1-P_\cV)(\cA)} \ .
\end{equation} 
Moreover, given a subspace $\cE$ 
such that $\cV\inc \cE \inc \cS$, 
then there exists a subspace $\cD\inc \cA$ such that 
\begin{equation}\label{eq desc sum2}
\cE=\cV+\cD=\cV\oplus (1-P_\cV)\cD\ .
\eeq
\end{pro}
\proof See Section \ref{Appendixity1}. \QED

\pausa
Now we can present the proof of our first main result.

\proof[Proof of Theorem \ref{teo opti 1}]
Recall that we consider an arbitrary square matrix $A\in\C^{n\times n}$ 
and its $A$-invariant subspace $\cX$, with $\dim\cX=d$. We further consider
a subspace $\cV\subset \C^n$, with  $A(\cV)\not\subset \cV$ and $\dim\cV=r\geq d$.
Finally, we let $V\in \C^{n\times r}$ have orthonormal columns such that
$R(V)=\cV$ and consider $R=AV-V(V^*AV)=(1-P_{\cV})AV$. 

\pausa 
We let $\cS=\cV+A(\cV)$, $\cM=P_\cS(\cX)\subseteq \cS$ and $\cN = (1-P_\cV)(\cM)$, so that 
 $\dim \cN  \leq \dim \cM \leq d$.  
By the first part of Proposition \ref{rem desc sum y algo mas}, 
\beq\label{M y N}
\cV+\cM=\cV\oplus (1-P_\cV)(\cM) = \cV\oplus\cN \inc \cS\,.
\eeq
Clearly, the range of the matrix $R=AV-V(V^*AV)=(1-P_\cV)AV$ coincides with $(1-P_\cV)A(\cV)$.
 Since $\cS=\cV\oplus (1-P_\cV) A(\cV)$ and 
$RR^\dagger$ is the orthogonal projection onto the range of $R$,
we conclude, once again by the first part of Proposition \ref{rem desc sum y algo mas}, that $P_\cS=P_\cV+RR^\dagger$.
Hence, $$\cN=(1-P_\cV) P_\cS(\cX)= (1-P_\cV) (P_\cV+RR^\dagger) (\cX)= RR^\dagger (\cX) 
$$ 
and $\cC_{\rm op}=R^\dagger(\cN)=R^\dagger RR^\dagger(\cX) = R^\dagger(\cX)$, 
where we used that $(1-P_\cV)P_\cV=0$ and $(1-P_\cV)R= R$. 
 These facts, and Eq. \eqref{eq main3}, show item 2.

\pausa
Note that $\cC_{\rm op} = R^\dagger (\, \cN\,) = R^\dagger (\cX)\inc \ker R^\perp \inc\C^r$, so $\dim\cC_{\rm op}\leq d$ and 
$ R(\cC_{\rm op}) = RR^\dag (\cN) =
\cN $.
Now, take a subspace  $\cC_{0}\subseteq \C^r$ such that 
$$
\dim \cC_{0}=d \py 
\cC_{\rm op}\subseteq P_{\ker R^\perp} \cC_{0}\,.
$$ 
For example, we can take $\cC_{0}$ to be any subspace with dimension $d$ that contains $\cC_{\rm op}$.
We set $\cW_{0}:=V\cC_{0} \subseteq \cV$ so that, by construction, $\dim \cW_{0}=d$ and 
$$ 
\cN=R(\cC_{\rm op})\subseteq RP_{\ker R^\perp}\cC_{0}=R(\cC_{0})
=(1-P_\cV)AV(\cC_{0})=(1-P_\cV)A (\cW_{0})\,.$$
Then,
$$
\cV+\cM \stackrel{\eqref{M y N}}=
\cV\oplus \cN 
\subseteq \cV\oplus(1-P_\cV)A (\cW_{0})=\cV+A (\cW_{0})\,.
$$ 
Therefore, $\cM\subseteq \cV+A (\cW_{0})\inc \cS$ and $\dim \cV+A (\cW_{0})\geq \dim\cV\geq d$.
 Let $\cW\subseteq \cV$ be such that $\dim\cW=d$; hence, $\dim \cV+A(\cW)\geq \dim \cV\geq d$ and then, by Proposition \ref{pro monotonia de angulos}, we have that 
$$
\Theta(\cX\coma\cV+A(\cW_0)) =\Theta(\cX\coma \cS) 
=\Theta(\cX\coma\cV+A(\cV)) \leq \Theta(\cX\coma\cV+A(\cW))\,.
$$ 
 This 
 shows item 1 and one implication in item 3.

\pausa
 Conversely,  assume that $\cW\subset \cV$ is such that $\dim \cW=d$ and  
$$
\Theta(\cX\coma\cV+A(\cW))=
\Theta(\cX\coma\cS)\,.
$$
By item $3.$ in Proposition \ref{pro monotonia de angulos}, 
we see that
$
\cM\subseteq \cV+A(\cW)$
so then $ \cV+\cM\subseteq \cV+A(\cW)\,.
$
If we let $\cC\subseteq \C^r$ be such that $\cW=V\cC$ then, 
$\dim \cC = \dim \cW  =d$ since $V$ has orthonormal columns.
By Proposition \ref{rem desc sum y algo mas} we conclude that
$$
R(\cC_{\rm op})= \cN=(1-P_\cV) (\cM)\subseteq (1-P_\cV)A(\cW)= (1-P_\cV) AV (\cC)=R(\cC)
\,. 
$$ 
Therefore 
 $\cC_{\rm op}=R^\dagger R (\cC_{\rm op})\subset R^\dagger R (\cC)=P_{\ker R^\perp}\cC$. So that 
$\cC$ is one of those subspaces of Eq. \eqref{el W0}. These facts show the other implication in item 3.

\pausa
Assume further that $\cV\cap A^{-1}(\cV)=\{0\}$ and that $\dim \cN=d$. The first of these assumptions implies that $\ker(R)=\{0\}$. Indeed, if there is a vector $x\in\C^r$ such that $0=Rx=(1-P_\cV)AVx$ then we must have $Vx\in A^{-1}(\cV)\cap\cV=\{0\}$. Hence $Vx=0$ and thus, $x=0$ since $V$ is injective. 

\pausa
Since $\cN$ and $\cC_{\rm op}$ have the same dimension (which we are assuming is $d$) and $R$ is injective, it follows that $\cC_{\rm op}$ is the only subspace that satisfies Eq. \eqref{el W0}. So, $\cW_0$ as in Eq. \eqref{los W0} is unique and given by $\cW_0=V(\cC_{\rm op})=VR^\dagger(\cX)$. Finally, using Proposition \ref{rem desc sum y algo mas} again, we have that
\[
\cV+A(\cW_0)
=
\cV\oplus(1-P_\cV)A(VR^\dagger(\cX))
=
\cV\oplus RR^\dagger(\cX)
=
\cV\oplus \cN
\,.
\]
\QED

\proof[Proof of Corollary \ref{coro para aplic}] We keep using the notation from the previous proof.
By Proposition \ref{rem desc sum y algo mas} we get that 
$$\cV+P_\cS(\cX)=\cV+(1-P_\cV)(P_\cS(\cX))=\cV\oplus \cN\,.$$
Item 1 follows from these facts.
Moreover, $P_\cS(\cX)\inc \cV\oplus \cN$ and hence, by item $2.$
 in Proposition \ref{pro monotonia de angulos}, $\Theta(\cX,\cS)=\Theta(\cX,\cV\oplus \cN)$ which proves the first part of item 3.
The proof of item 2. and the rest of item 3. can be obtained using Proposition \ref{rem desc sum y algo mas} in a way similar to that considered in the proof of Theorem \ref{teo opti 1} above.

\pausa
To show item 4., let $\cW_0\subset \cV$ with $\dim \cW_0=d$ satisfying Eq. \eqref{los W0}. By item 3. in Theorem \ref{teo opti 1} $\cV\oplus \cN\subset \cV+A(\cW_0)$. Since $\dim \cV\oplus\cN=r+d\geq \dim \cV+A(\cW_0)$ we conclude that 
$\cV\oplus \cN=\cV+A(\cW_0)$.
\QED

\pausa
{\bf Proofs of the results in Section \ref{sec main2}}. To obtain a detailed proof of 
Theorems \ref{teo main conv vt} and \ref{teo main Kt} we will describe several technical results. To simplify our exposition we consider the following
\begin{nota}\label{notac2} Let $A\in\C^{n\times n}$ be an Hermitian matrix. In what follows we consider:
\begin{enumerate}
\item An eigen-decomposition $A=X\Lambda X^*$ , where $X\in \C^{n\times n}$ is a unitary matrix and 
$\Lambda=\text{diag}(\la_1,\ldots,\la_n)$, with $\la_1\geq \ldots\geq \la_n$. We let $x_1,\ldots,x_n$ denote the columns of $X$. Hence, $A\,x_j=\la_j\,x_j$, for $1\leq j\leq n$. We assume that $\la_d>\la_{d+1}$ and set $\cX=\text{span}\{x_1,\ldots,x_d\}$.
\item For $1\leq \ell\leq n$ we consider the partitions:
$$
X=\begin{bmatrix} X_\ell& X_{\ell,\bot} \end{bmatrix} \peso{and} \Lambda=\begin{bmatrix} \Lambda_\ell & 0 \\ 0 & \Lambda_{\ell,\bot} \end{bmatrix} \,.
$$ where $X_\ell\in \C^{n\times \ell}$ and $\Lambda_\ell\in\C^{\ell\times \ell}$. 
\end{enumerate}
\end{nota}

\begin{rem}
Consider Notation \ref{notac2}. Since we are assuming that $\la_d>\la_{d+1}$ then, 
the subspace $\cX$ as in Theorems \ref{teo main conv vt} and \ref{teo main Kt} is uniquely determined. Moreover, we have that $\cX=R(X_d)$. \EOE
\end{rem}

\pausa
Consider Notation \ref{notac2}. Let $V\in\C^{n\times r}$ be such that $R(V)=\cV$ with $\dim \cV=r\geq d$ (but $V$ does not necessarily have orthonormal columns). We are interested in considering the singular values of the expression $X_{d,\bot}^*V(X_d^*V)^\dagger$ for arbitrary $V$ as above. 
On the other hand, in \cite[Theorem 3.1]{KZ13} Zhu and Knyazev show 
that if $\cX\cap \cV^\perp=\{0\}$ (or, equivalently, $\rk(X^*V)=\rk(V^*X)=d$) and we further assume that $V$ has orthonormal columns 
 then, the first $d$ singular values of the matrix $X_{d,\bot}^*V(X_d^*V)^\dagger$ are the  tangents of the angles between $\cX$ and $\cV$ (that necessarily lay in $[0,\,\pi/2)$).

\pausa
The following result plays a key role in our analysis.

\begin{teo}\label{cor: tang vs raro} Consider Notation \ref{notac2}
and let $V\in\C^{n\times r}$ and $\cV=R(V)$ be such that $\dim \cV=r\geq d$ and $\cX\cap R(V)^\perp=\{0\}$. 
 Then, for every unitary invariant norm $\|\cdot\|$ we have that
\begin{equation}\label{eq nui}
\|\tan(\Theta(\cX,\cV))\|
\leq
\|X_{d,\bot}^* V(X_d^*V)^\dagger\|\,.
\end{equation}
\noindent 
\end{teo}
\proof See Section \ref{subsec angles}.
\qed

\pausa
We remark that the inequalities in Eq. \eqref{eq nui} can be strict, even when $\dim R(V)=r$ (i.e. when $V\in\C^{n\times r}$ has linearly independent columns); see Section \ref{subsec angles}. Using Theorem \ref{cor: tang vs raro} we can deduce the following 

\begin{teo}\label{teo cota con polinomio}
    Consider Notation \ref{notac2}. Let $\cV\subset \C^n$ with $\dim\cV=r\geq d$,   
		 fix $d\leq p\leq r$ and assume that $\text{span}\{x_1,\ldots,x_p\}\cap \cV^\perp=\{0\}$. Then, there exists $\cH_p\subset \cV$ with $\dim\cH_p=d+r-p$ such that: $\cX\cap \cH_p^\perp=\{0\}$ and for every polynomial $\phi\in \C[x]$ such that $\phi(\Lambda_d)$ is invertible and for every unitarily invariant norm $\|\cdot\|$ we have that
		\begin{equation}\label{eq nui 2}
        \|\tan(\Theta(\cX,\, \phi(A)(\cV) ))\|
        \leq
         \|\phi(\Lambda_d)^{-1}\|_2\ \|\phi(\Lambda_{p,\perp})\|_2
        \;\| \tan(\Theta(\cX,\, \cH_p))\|
        \,.
    \end{equation}
\end{teo}
\proof See Section \ref{subsec algos}.
\qed

\pausa We point out that the hypothesis $\text{span}\{x_1,\ldots,x_p\}\cap \cV^\perp=\{0\}$ holds in a generic case. On the other hand, the condition $\cX\cap \cH_p^\perp=\{0\}$  implies that the tangents of the principal angles $\Theta(\cX,\cH_p)$, that appear in the right-hand side of Eq. \eqref{eq nui 2}, are well defined.

\pausa
The proof of Theorem \ref{teo cota con polinomio} partially relies on 
an strategy to amplify singular gaps by considering a 
convenient subspace of $\cV$, first developed by Gu \cite{Gu} (see also \cite{WWZ}). 
The proofs of Theorems \ref{teo main conv vt} and \ref{teo main Kt} follow from Theorem \ref{teo cota con polinomio} together with convenient choices of polynomials (depending on the algorithm under consideration).  The following result provides such convenient choices.

\begin{pro}\label{pro estimac poly main}
Consider Notation \ref{notac2}, let $1\leq t$ and $d\leq p\leq n-1$.
\begin{enumerate}
\item Set $\phi(x)=x-(\la_{d+1}+\la_n)/2$. Then 
$$\|\phi(\Lambda_d)^{-1}\|_2\,\|\phi(\Lambda_{d,\bot})\|_2 =    \frac{\lambda_{d+1}-\lambda_n}{(\lambda_d-\lambda_n)+(\lambda_d-\lambda_{d+1})}<1\,.$$
\item Let $T_t(x)$ denote the Chebyshev polynomial of the first kind of degree $t\geq 1$ and set $\phi_t(x)=T_t((x-\la_n)/(\la_{p+1}-\la_n))$. Then 
$$
\|\phi_t(\Lambda_d)^{-1}\|_2 \; \|\phi_t(\Lambda_{p,\perp})\|_2 \ 
\leq 
4\, \frac{\la_{p+1}-\la_n}{\lambda_d-\la_n} \ 3^{-t\cdot\min\left\{\sqrt{\frac{\la_d-\la_n}{\la_{p+1}-\la_n} \,-\,1}\,,\,1\right\}}\,.
$$
\end{enumerate}
\end{pro}
\proof
See Sections \ref{subsec algos} and \ref{subsec cheby}.
\QED

\pausa
We can now present the following

\begin{proof}[Proof of Theorem \ref{teo main conv vt}] We consider Notation \ref{notac2} and let
$\{\cV_t\}_{t=0}^q$ be constructed as in Algorithm \ref{algo1}. We assume further that $\cX\cap \cV^\perp=\{0\}$ and consider a unitarily invariant norm $\|\cdot \|$ on $\C^{n\times n}$.

\pausa
We fix $1\leq t\leq q$ and set $\phi(x)=x-s$, for $s\notin \{\lambda_1,\ldots,\lambda_d\}$. 
Notice that $\phi(A)(\cV_{t-1}) \subseteq \cV_{t-1}+A(\cV_{t-1})$; recall that the subspace $\cV_t$ satisfies that $\Theta(\cX,\,\cV_t) = \Theta(\cX,\,\cV_{t-1}+A(\cV_{t-1}))$. Combining these facts with item 1. in Proposition \ref{pro monotonia de angulos} we obtain that 
    \[
    \|\tan(\Theta(\cX,\,\cV_t))\|
    =
    \|\tan(\Theta(\cX,\,\cV_{t-1}+A(\cV_{t-1})))\|
    \leq
    \|\tan(\Theta(\cX,\,\phi(A)(\cV_{t-1})))\|
    \,.
    \]
    \noindent 
    
    \noindent Since $\cV_0\subseteq\cV_{t-1}$, the condition  $\cX\cap\cV_0^\bot=\cX\cap\cV^\bot=\{0\}$ implies that $\cX\cap\cV_{t-1}^\bot=\{0\}$. Hence, we can apply Theorem \ref{teo cota con polinomio} with the subspace $\cV_{t-1}$ and $p=d$ (so that $\cH_d=\cV_{t-1}$) and get that 
    \[
    \|\tan(\Theta(\cX,\,\phi(A)(\cV_{t-1})))\|
    \leq
\|\phi(\Lambda_d)^{-1}\|_2 \    \|\phi(\Lambda_{d,\perp})\|_2 \ 
        \| \tan(\Theta(\cX, \cV_{t-1}))\| 
    \,. \]
  Taking into account the previous two inequalities and choosing $s=(\lambda_{d+1}+\lambda_n)/2$  then, by 
 item 1. in Proposition \ref{pro estimac poly main} we get that		
    \begin{align*}
        \|\tan(\Theta(\cX,\,\cV_t))\|
        &\leq
\frac{\lambda_{d+1}-\lambda_n}{(\lambda_d-\lambda_n)+(\lambda_d-\lambda_{d+1})}
        \;\| \tan(\Theta(\cX,\, \cV_{t-1}))\|\,.
    \end{align*}

    \noindent Eq. \eqref{eq hay conv1} follows by applying iteratively the previous estimation. 
\end{proof}

\begin{proof}[Proof of Theorem \ref{teo main Kt}]
Consider Notation \ref{notac2} and let $\cV\subset \C^n$ be such that $\dim \cV=r\geq d$ and $\text{span}\{x_1,\ldots,x_r\}\cap \cV^\perp=\{0\}$. Let $\{\cK_t\}_{t=0}^q$ be the finite sequence of block Krylov subspaces 
induced by $\cV=\cK_0$.  By Theorem \ref{teo cota con polinomio} there is a subspace $\cH_p\subset \cV$ with $\dim\cH_p=d+r-p\geq d$, $\cX\cap \cH_p^\perp=\{0\}$ and such that 
    \[
    \|\tan(\Theta(\cX,\,\phi(A)(\cV)))\|
    \leq
     \|\phi(\Lambda_d)^{-1}\|_2\ \|\phi(\Lambda_{p,\perp})\|_2 
    \;\| \tan(\Theta(\cX,\, \cH_p))\|
    \,.
    \]
		    If we let $\phi(x)\in \C[x]$ be a polynomial of degree at most $t$ then
we get that $\phi(A)(\cV)\subset \cK_t$; this last fact, together with item 1. in
				 Proposition \ref{pro monotonia de angulos} show that
    \[
    \|\tan(\Theta(\cX,\,\cK_t))\|
    \leq
    \|\tan(\Theta(\cX,\,\phi(A)(\cV)))\|
    \leq
     \|\phi(\Lambda_d)^{-1}\|_2\ \|\phi(\Lambda_{p,\perp})\|_2 
    \;\| \tan(\Theta(\cX,\, \cH_p))\|\,.
    \]
The result now follows from the previous facts and item 2. in Proposition \ref{pro estimac poly main}.
\end{proof}

\section{Numerical examples}\label{sec num exa}

In this section we consider some numerical examples related to the results in Section \ref{sec main results general}.
These have been performed in Python (version 3.8.1) mainly using numpy (version 1.24.4) and scipy (version 1.10.1) packages with machine precision of $2.22\times10^{-16}$. 
We first describe the different matrices and parameters considered in our examples. We then elaborate on different aspects of 
the subspace approximations derived from Algorithm \ref{algo1} (optimal theoretical subspace expansion)
and its numerical counterpart, namely Algorithm \ref{algo2} (computable subspace expansion) for the Rayleigh-Ritz (RR) and refined Rayleigh-Ritz (refRR) methods. We have also included implementations of the block Krylov method and the numerical methods described in Remark \ref{rem comparacion}. Finally, we also show plots of the upper bounds derived from Theorems \ref{teo main conv vt} and \ref{teo main Kt}.

\pausa
Throughout this section we keep using the notation of Section \ref{sec main results general}. In particular, $d$ stands for the dimension of the target space $\cX\subseteq\mathbb{C}^n$, $r\geq d$ stands for the dimension of the starting guess subspace $\cV\subseteq\mathbb{C}^n$ and $d\leq p\leq r$ stands for the parameter considered Theorem \ref{teo main Kt}. Indeed, the subspaces $\cV$ were constructed as the range of matrices $V\in \mathbb R^{n\times r}$ drawn from an $n\times r$ standard Gaussian random matrix, with $r$ according to the cases described below. 
We notice that the upper bounds obtained in this work depend on expressions of the form $\|\tan(\Theta(\cX,\,\cV))\|$. It will be shown in Theorem \ref{teo: tang vs raro} that this expression is bounded by $\| X_{\perp}^* V \, (X^*\,V)^\dagger\|$ where $X$ has orthonormal columns that span $\cX$ and $X_\perp$ has orthonormal columns that span $\cX^\bot$ (since the condition $\dim X^*(\cV)=d$ holds with probability 1).
It is known that the last expression can be controlled with high probability for the spectral and Frobenius norms (see for example \cite{Gu, Saibaba}) although we have not followed that research direction in this notes. 

\pausa
In all the examples shown below we have set $n=5000$ and the target dimension $d$ to 5, but we have also obtained similar results in the case where ($n=5000$ and) $d$ is set to $15$. We have set $60$ as the standard number of iterations for our experiments, but our implementations of the Algorithm \ref{algo2} for the RR projection method, and the numerical algorithms from Remark \ref{rem comparacion} are considerably faster than our implementations of the block Krylov algorithm and Algorithm \ref{algo2} for the ref RR projection method for this value of $q$ (specially for higher values of the parameter $r$) and thus we show these faster methods with higher number of iterations. 

\pausa
{\bf Test matrices:} In all the examples considered below we have set $A$ to be a square diagonal matrix of size $5000$, where the diagonal entries are non negative and show different kinds of decays. For any of the (baseline) decays, we exhibit the constant from Theorem \ref{teo main conv vt}, namely
\[
\mu_d
:=
\frac{\lambda_{d+1}-\lambda_n}{(\lambda_d-\lambda_n)+(\lambda_d-\lambda_{d+1})}<1\,,
\]
 which appears in our upper bound for the decay of the angles $\Theta(\cX,\,\cV_t)$ as a function of $t$. 

\pausa For the numerical experiments, we have considered the following:

\pausa 
1. {\bf Linear decay}. Our starting linear model is given by $A_{ii}=3000-\frac{3}{5}i$, for $1\leq i\leq 5000$. In this case $\mu_d\approx 0.9996$.  We point out that although simple, the linear decay is a rather challenging model.
In this setting we have considered different amounts of oversampling: $r=60$, $p=40$ (large); $r=30$, $p=20$ (moderate); $r=20$, $p=15$ (small). 
We also show variations by forcing a bigger singular gap at the prescribed index $d$. This was done by redefining: $A_{ii}=A_{ii}+G$ for $1\leq i\leq d$ and some $G\geq 0$.
For this model, we have set $G=0$ (baseline); $G=10$; $G=40$ and $G=70$. For these values of $G$, $\mu_d$ goes down to approximately $0.9929$, $0.9736$ and $0.955$ respectively.

\pausa
2. {\bf Ellipsoidal decay}. We considered 
$A_{ii}= b (1- (a + c (\frac{i}{n})^2) ) ^{1/2} - H$, for $1\leq i\leq 5000$.
Here the parameters were chosen so that $\mu_d$ remains close to 1 but the singular values are not too far apart. For the parameters that produced the plot shown in Figure \ref{fig: perfiles} we have $\mu_d\approx 0.9997$ and for the biggest value of $G$ considered we have $\mu_d\approx0.9872$. We experimented with different oversampling sizes and forced singular gaps in a similar manner to that of the linear decay. Since this model seems even more challenging than the previous one, our choices for the values $r$ and $p$ are generally bigger. Also, since the entries $A_{ii}$ are much smaller, the choices for $G$ are quite smaller.

\pausa
3. {\bf Polynomial decay}. Finally, we considered 
$A_{ii}= c / ((i+s)^{1/2} + m)$, for $1\leq i\leq 5000$ where  $c=1000$, $m=50$ and $s=200$, which gives $\mu_d\approx 0.9977$.
For this model we have also experimented with different amounts of oversampling and forced singular gaps and observed similar behaviors to those of the previous models, so we only report on two experiments where we have set $r=30$ and $p=20$ and $G=0$.

\begin{figure}[hb]
     \centering
     \includegraphics[width=\textwidth]{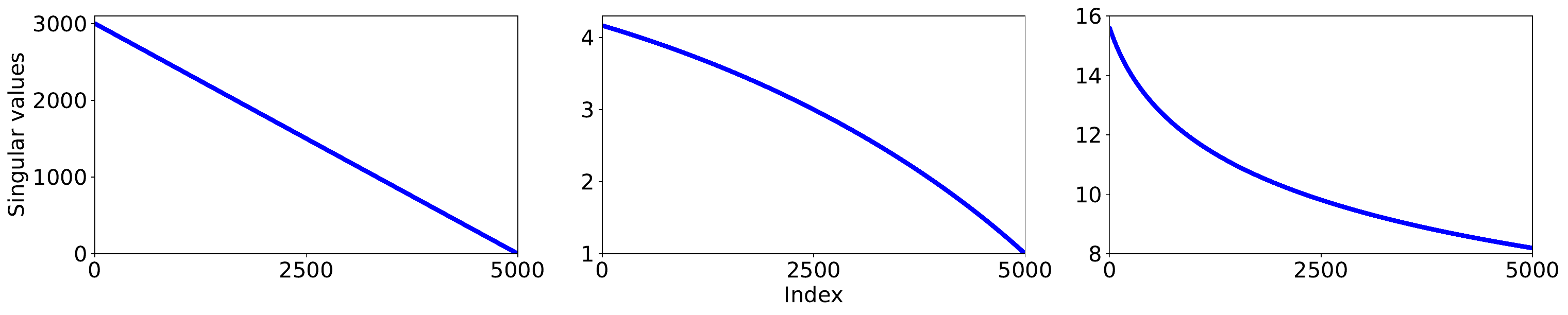}
        \caption{Plots of the different (baseline) diagonals considered for the test matrices.}
    \label{fig: perfiles}
\end{figure}

\pausa
{\bf Numerical results}. Given $A$ and a starting matrix $V$ as above with $R(V)=\cV$, we have implemented
the block Krylov algorithm for the construction of the subspaces $\cK_t=\cK_t(A,X)$ given recursively by:
$\cK_0=\cV$ and $\cK_{t}=\cK_{t-1} + A(\cK_{t-1})$, for $t\geq 1$. Similarly, we have implemented 
Algorithm \ref{algo1} for the computation of theoretical optimal expansions. We have also 
implemented Algorithm \ref{algo2} for the computation of the numerical counterparts 
based on the Rayleigh-Ritz (RR with $\cS$), refined Rayleigh-Ritz (refRR with $\cS$) methods; the extension of Jia's expansion method based on the Rayleigh-Ritz (RR with $\cR$) method and the Residual Arnoldi (RA) method from Remark \ref{rem comparacion} \cite{J94,J97,J98,Jia04,Jia22,SJia01,Lee07,LeeS07}.

In the following figures we show the decay of the biggest angle between the subspaces produced by the algorithms mentioned above and the target subspace $\cX$ (we use the labels $\theta(\cX,\,\cK_q)$ and $\theta(\cX,\,\cV_q)$). Notice that in all our test matrices the subspace $\cX$ is the subspace generated by the first $d$ vectors of the canonical basis. We also show upper bounds for those angles derived from Theorem \ref{teo main Kt} and Theorem \ref{teo main conv vt} by taking the unitarily invariant norm in those statements as the operator norm (we use the labels UB on $\theta(\cX,\,\cK_q)$ and UB on $\theta(\cX,\,\cV_q)$ respectively).
We have considered first the linear decay models (Figure \ref{fig: modelos lineales}), followed by the elliptical decay models (Figure \ref{fig: modelos elipticos}) and lastly, we consider two examples for the better-behaved polynomial decay models (Figure \ref{fig: modelos polinomicos}) accompanied with a graph that shows the dimensions of the subspaces $\cK_t$ and $\cV_t$ as a function of $t$ for two values of the parameter $r$ that we have used in the previous experiments.

\begin{figure}[ht!]

    \centering
    \begin{subfigure}[b]{0.325\textwidth}
         \includegraphics[width=\textwidth]{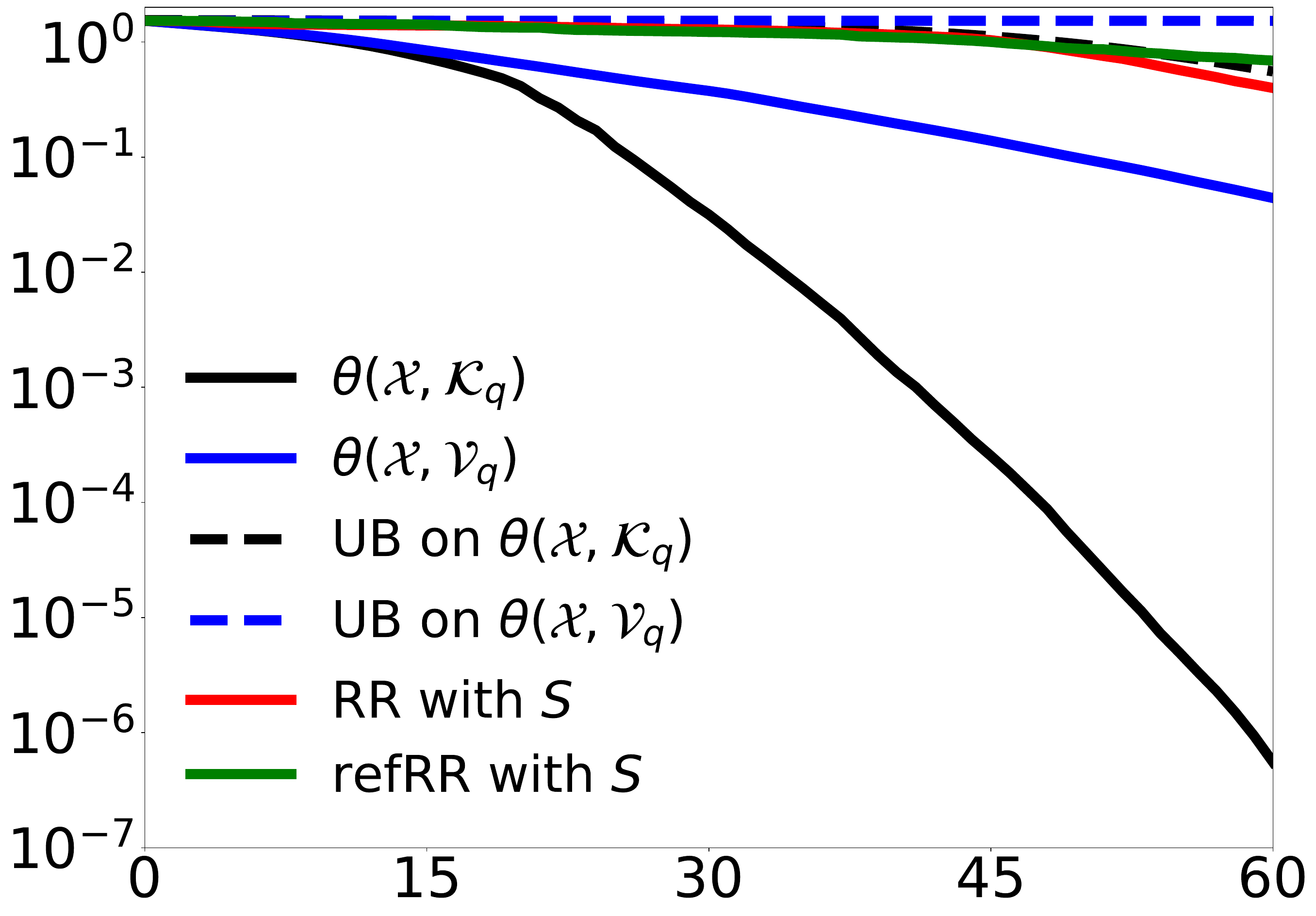}
         \caption{$r=30$, $p=20$ and $G=0$}
         \label{fig: lineal plano}
    \end{subfigure}
    \hfill
    \begin{subfigure}[b]{0.325\textwidth}
        \includegraphics[width=\textwidth]{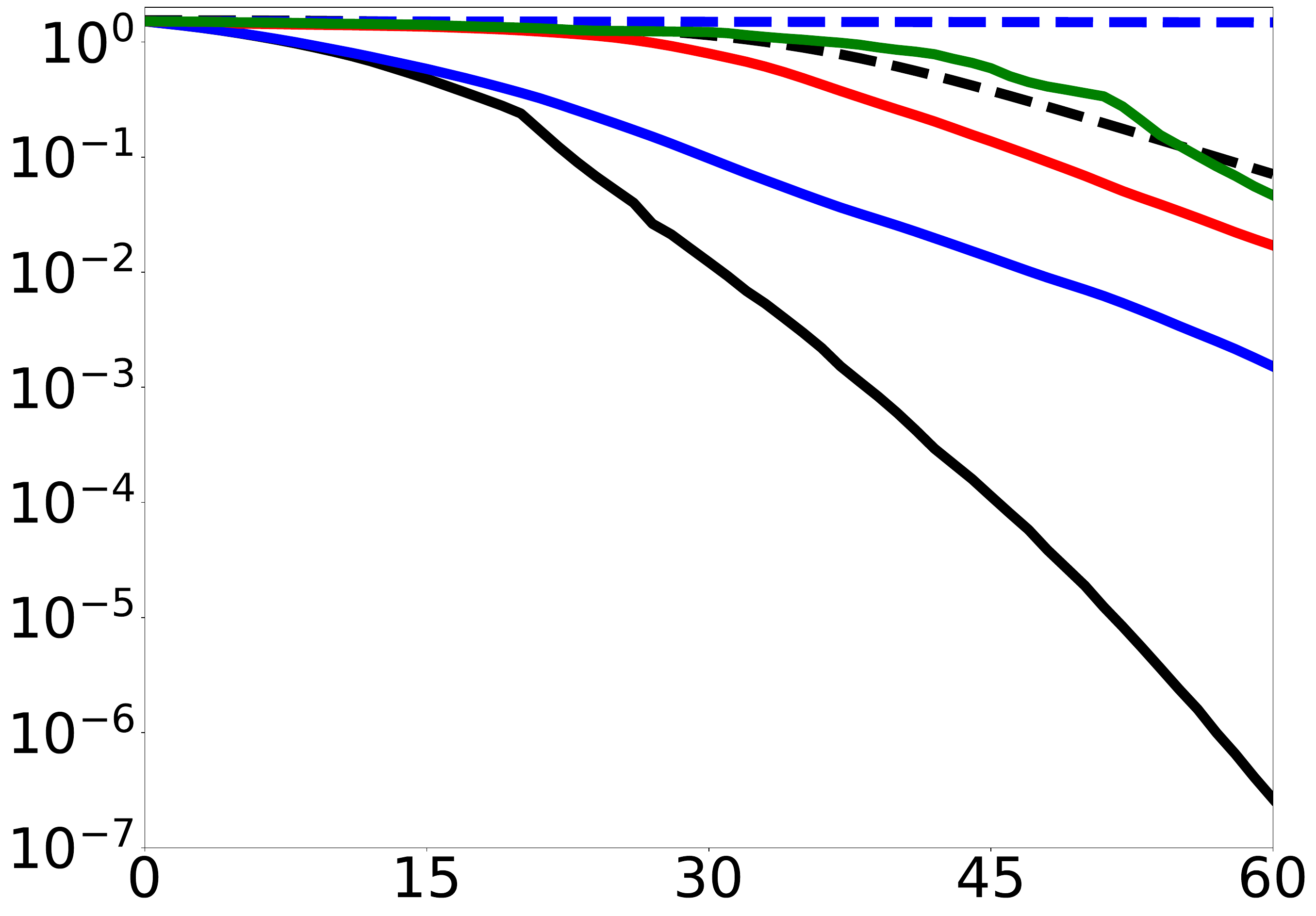}
        \caption{$r=30$, $p=20$ and $G=10$}
        \label{fig: lineal poco gap}
    \end{subfigure}
    \hfill
    \begin{subfigure}[b]{0.325\textwidth}
         \includegraphics[width=\textwidth]{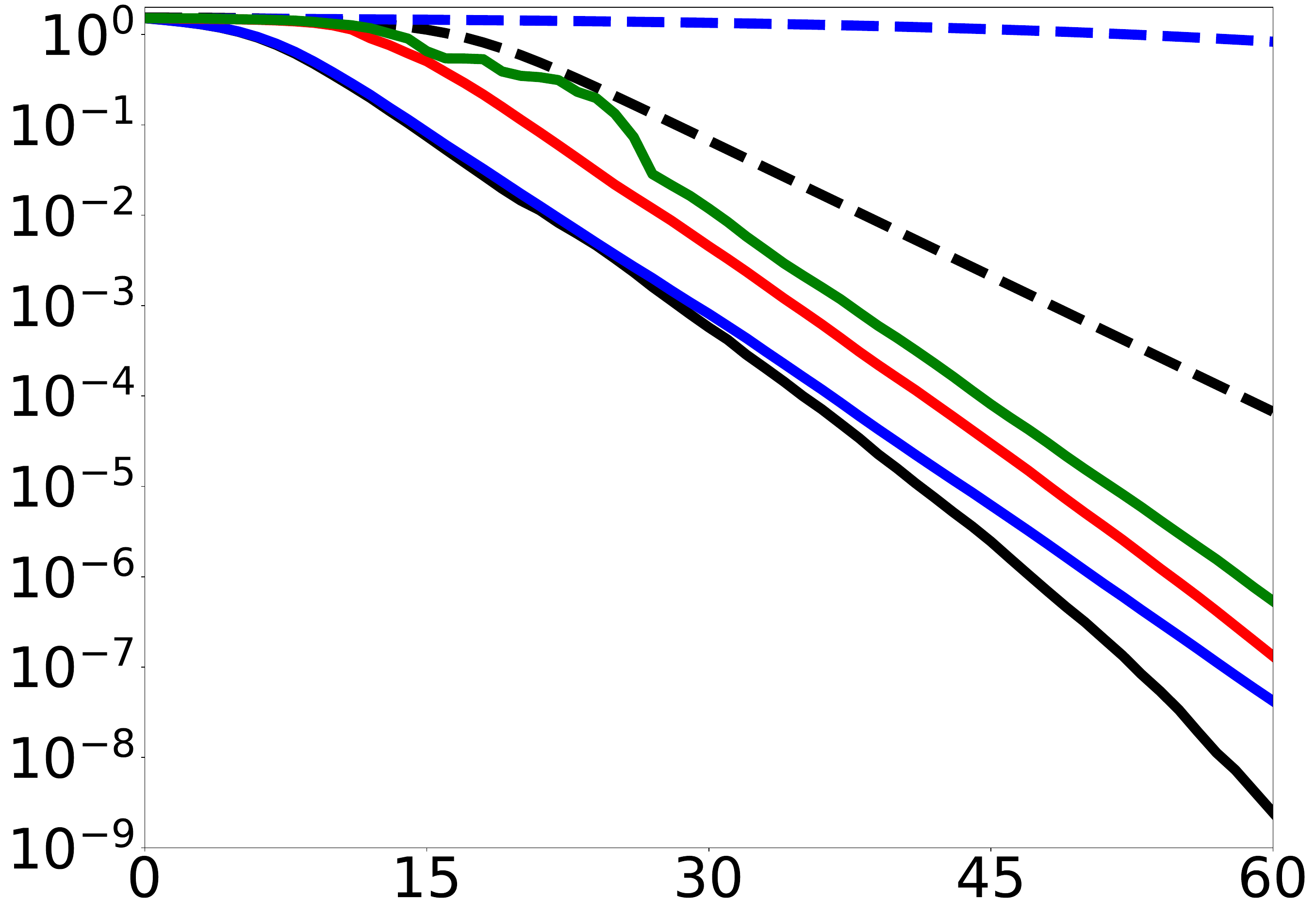}
         \caption{$r=30$, $p=20$ and $G=70$}
         \label{fig: lineal mucho gap}
    \end{subfigure}
        \vspace{0.5cm}
        \begin{subfigure}[b]{0.325\textwidth}
        \includegraphics[width=\textwidth]{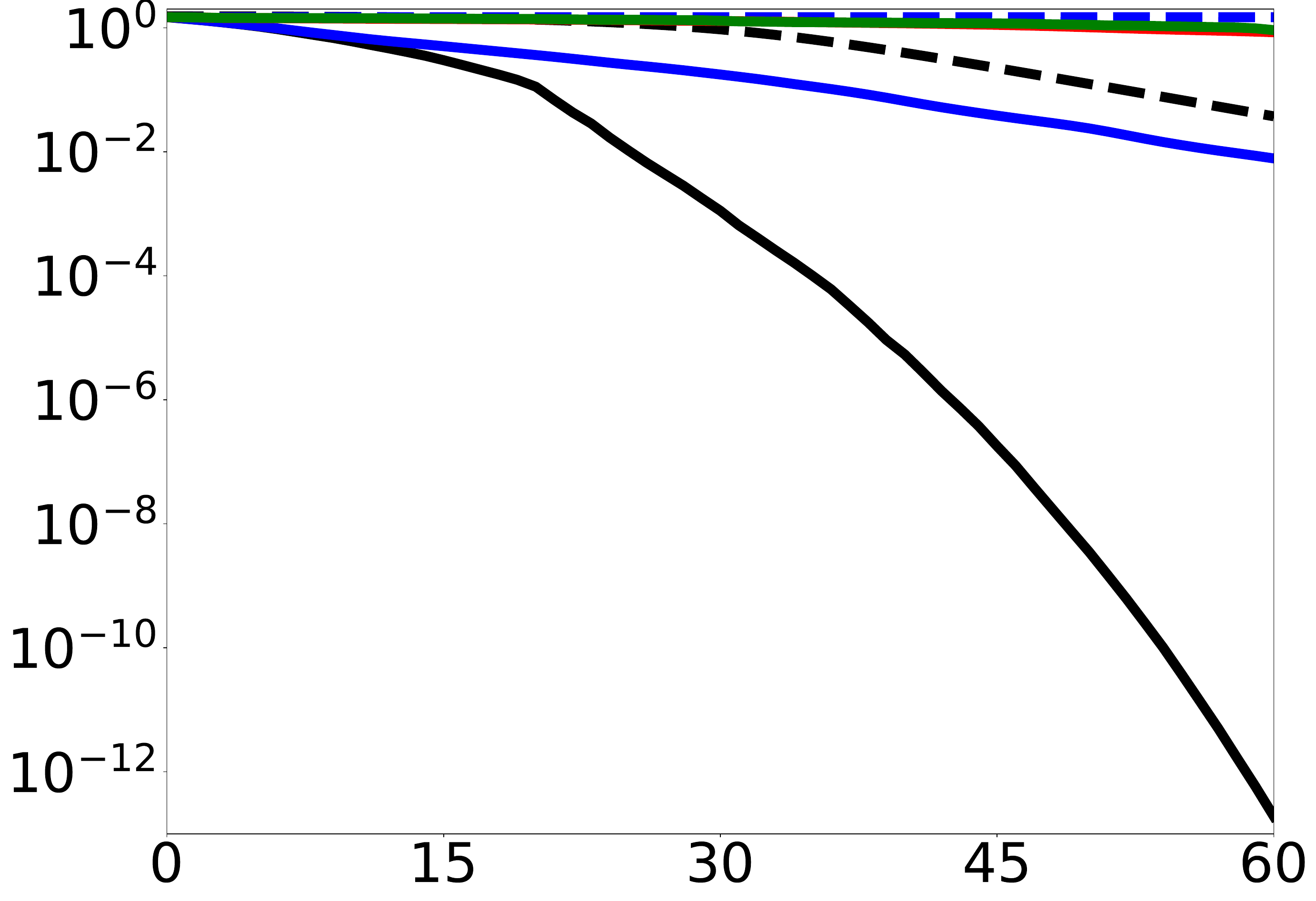}
        \caption{$r=60$, $p=40$ and $G=0$}
        \label{fig: lineal plano mucho OV}
    \end{subfigure}
        \hfill
        \begin{subfigure}[b]{0.325\textwidth}
        \includegraphics[width=\textwidth]{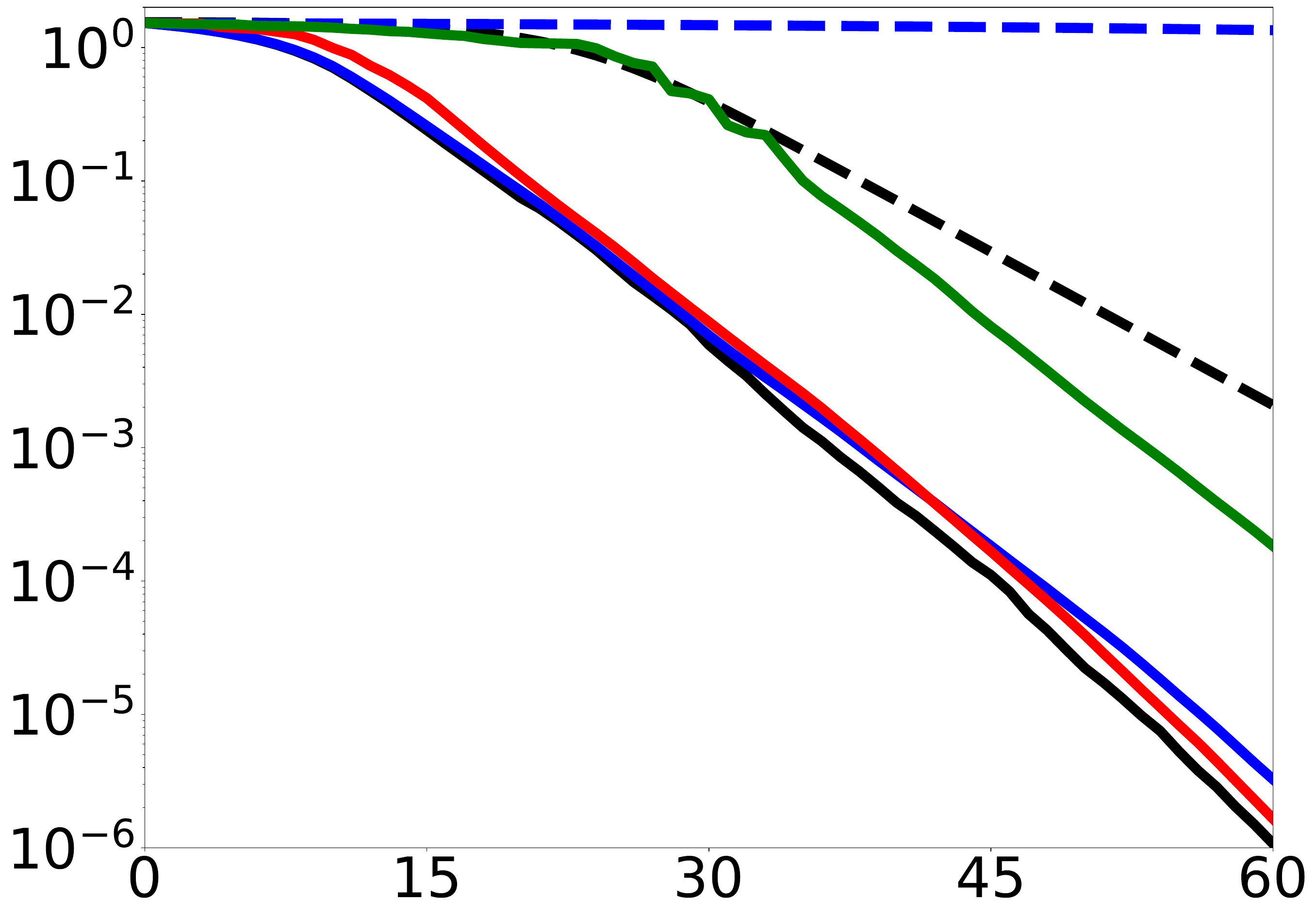}
        \caption{$r=20$, $p=15$ and $G=40$}
        \label{fig: lineal plano poco OS y gap}
    \end{subfigure}
       \hfill
   	\begin{subfigure}[b]{0.325\textwidth}
        \includegraphics[width=\textwidth]{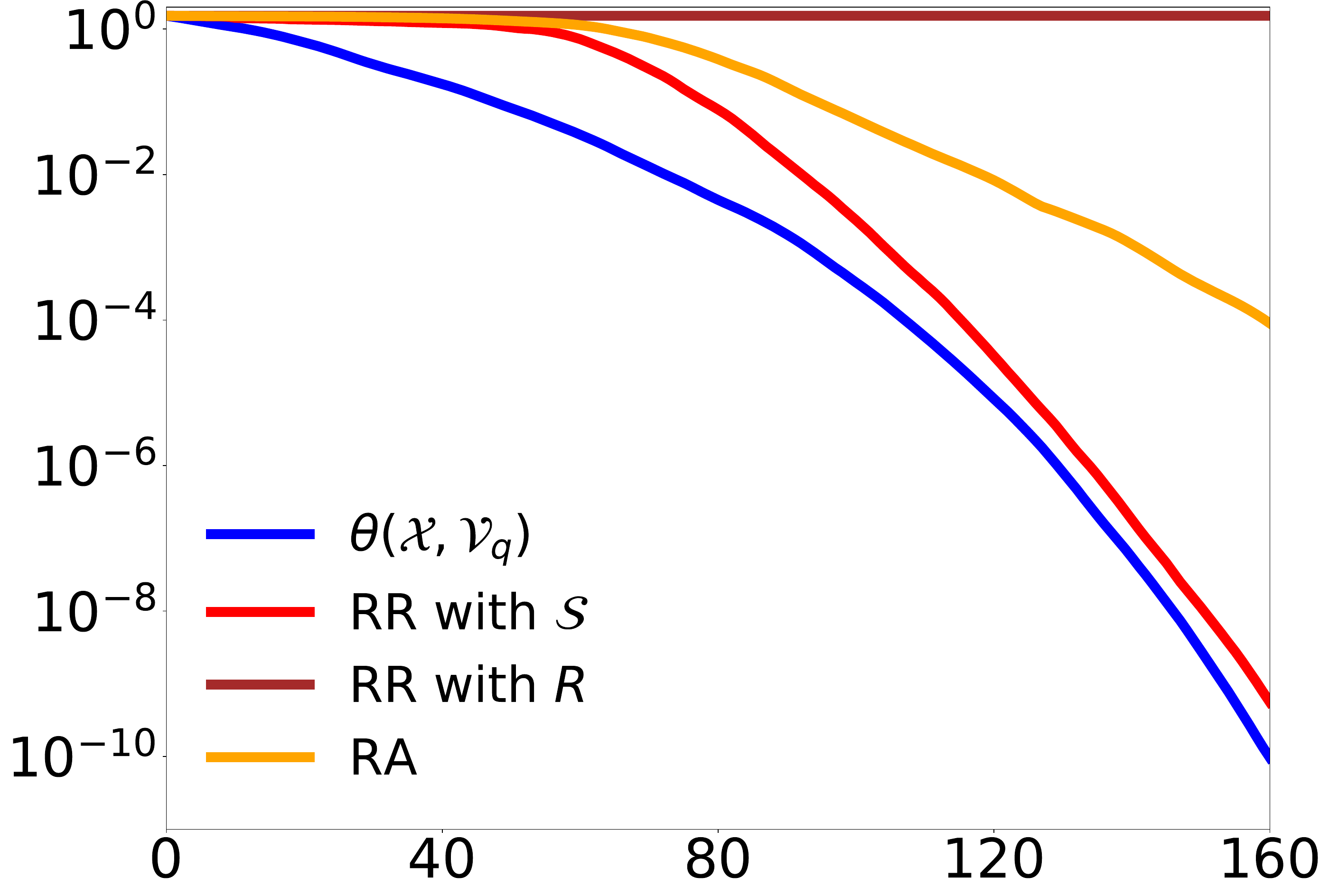}
        \caption{$r=30$, $p=20$ and $G=0$}
        \label{fig: lineal rapidos}
    \end{subfigure}
  \caption{Linear decay models}
  \label{fig: modelos lineales}
\end{figure}

\begin{figure}[ht!]
    \centering
  		\begin{subfigure}[b]{0.325\textwidth}
         \centering
         \includegraphics[width=\textwidth]{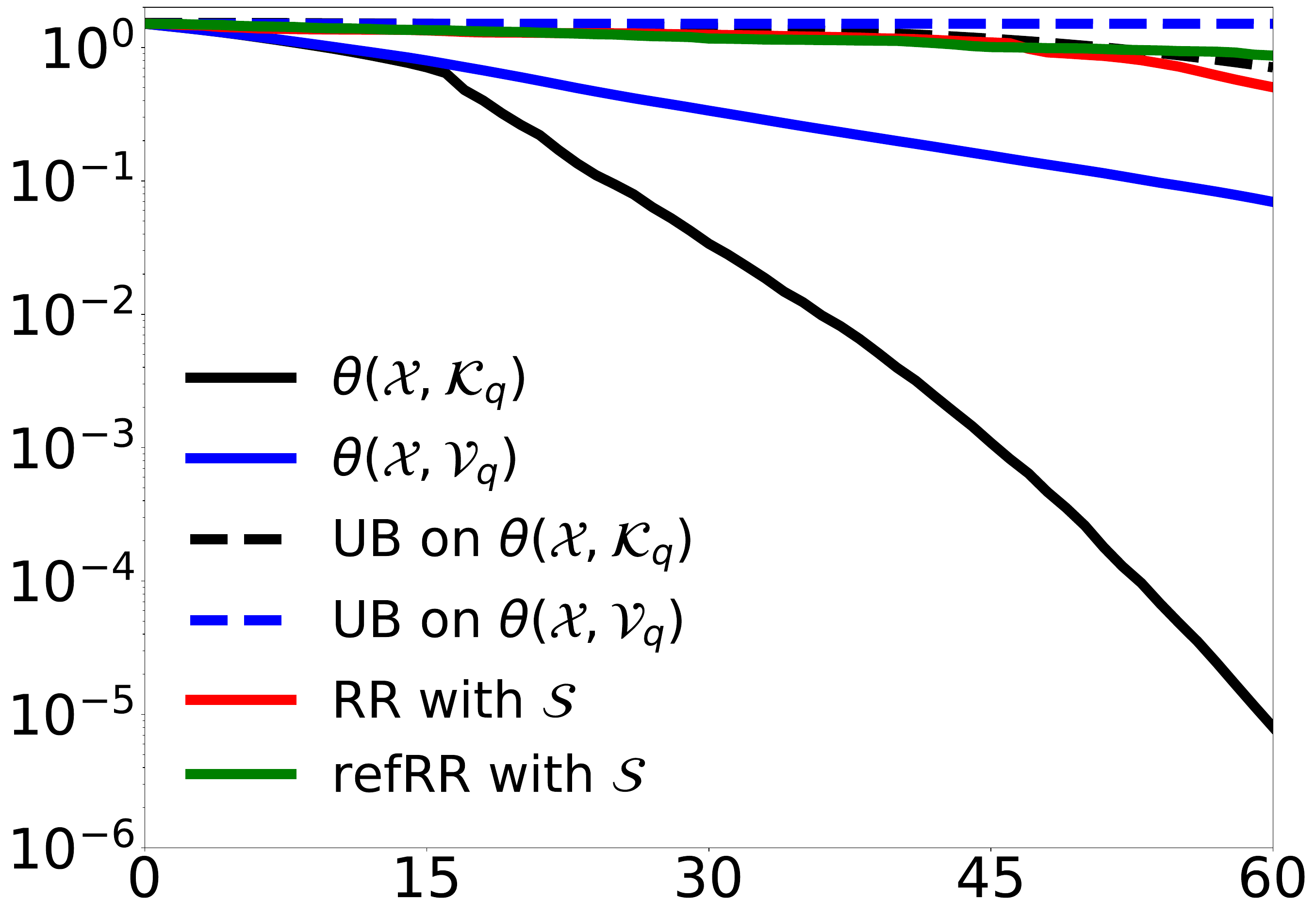}
         \caption{$r=45$, $p=30$ and $G=0$}
         \label{fig: Elip_1_plano}
     \end{subfigure}
        \hfill
   \begin{subfigure}[b]{0.325\textwidth}
         \centering
         \includegraphics[width=\textwidth]{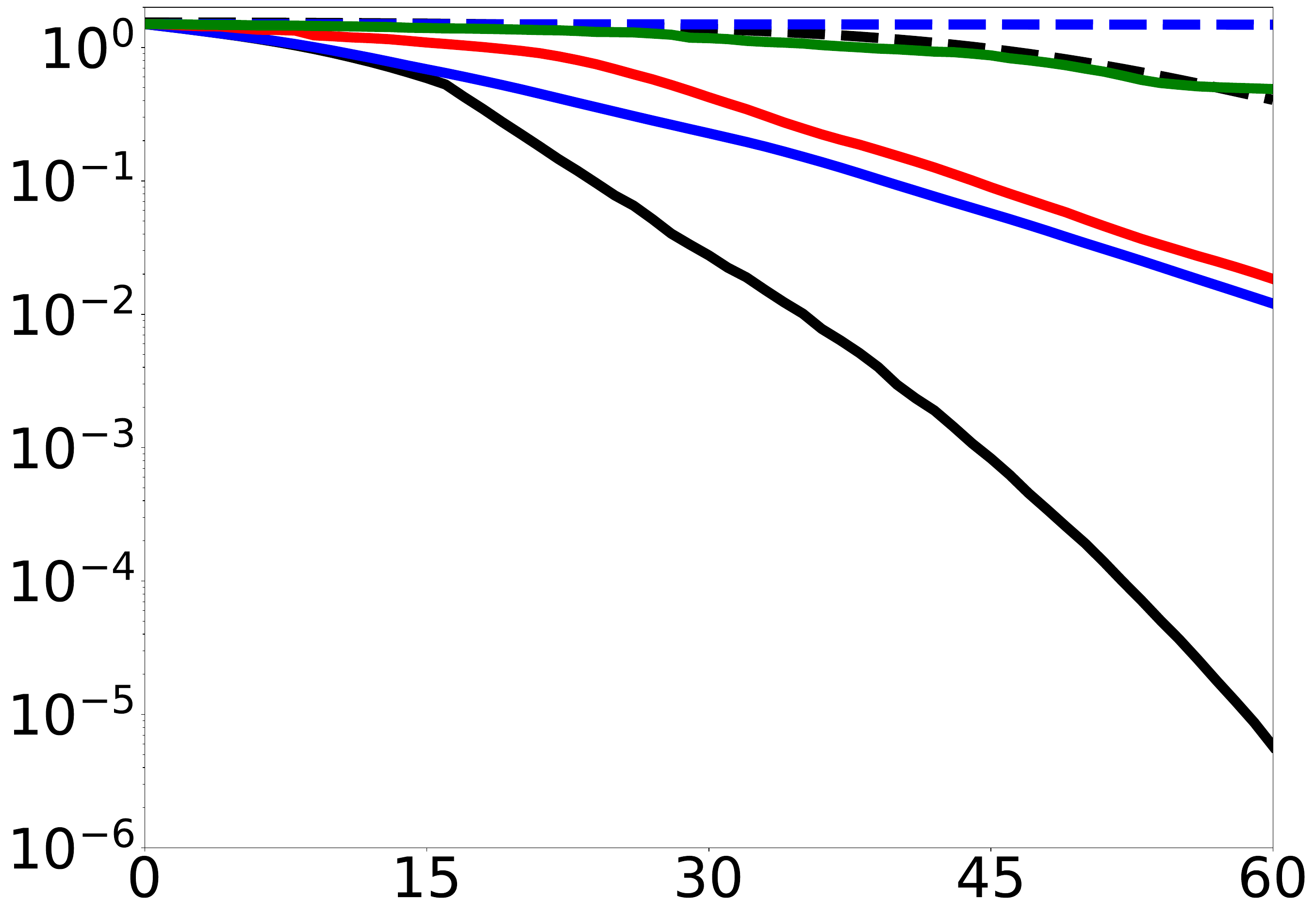}
         \caption{$r=45$, $p=30$ and $G=0.005$}
              \end{subfigure}
       \hfill
        \begin{subfigure}[b]{0.325\textwidth}
         \centering
         \includegraphics[width=\textwidth]{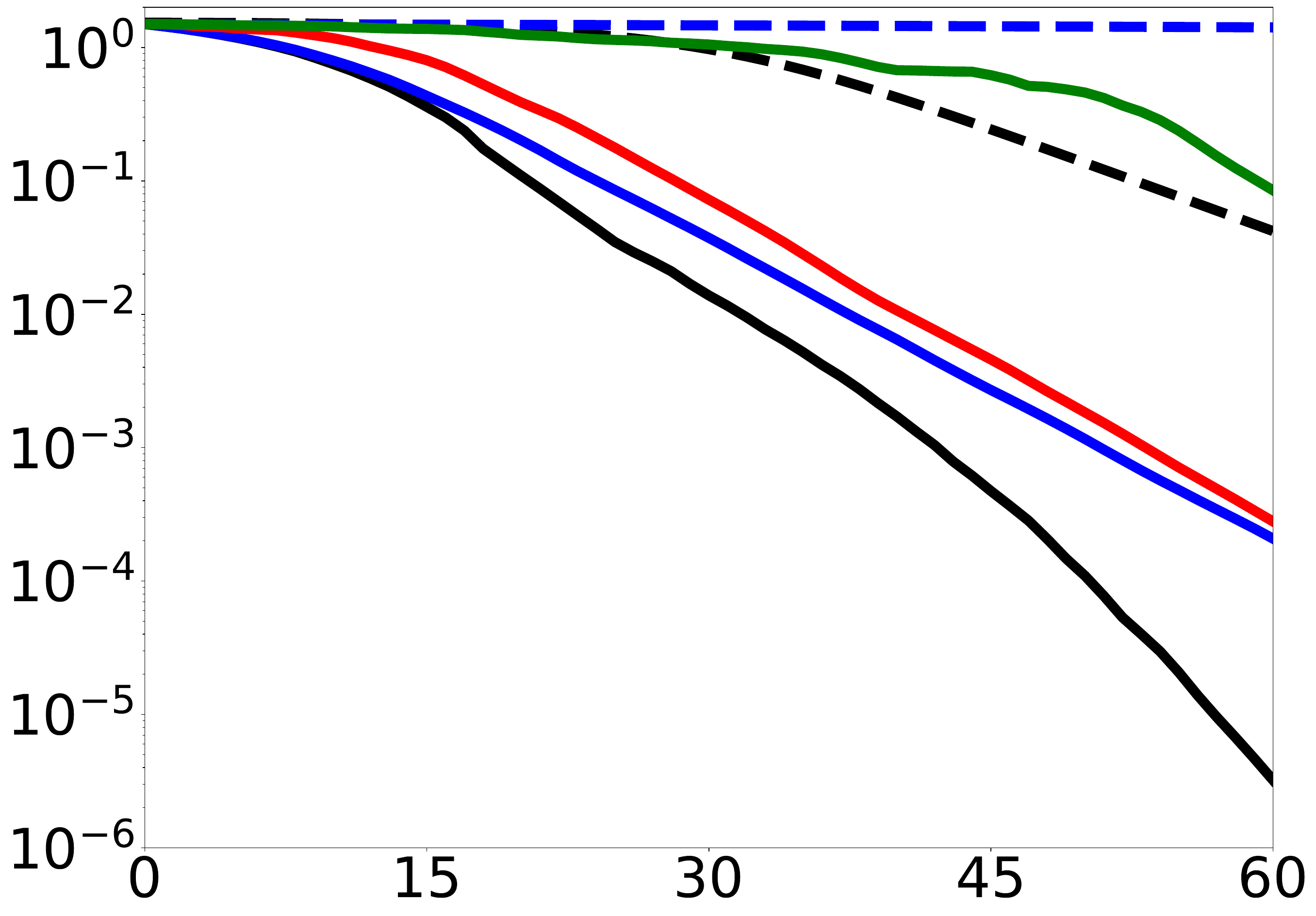}
         \caption{$r=45$, $p=30$ and $G=0.02$}
         \end{subfigure}
       \vspace{0.5cm}
    \begin{subfigure}[b]{0.325\textwidth}
         \centering
         \includegraphics[width=\textwidth]{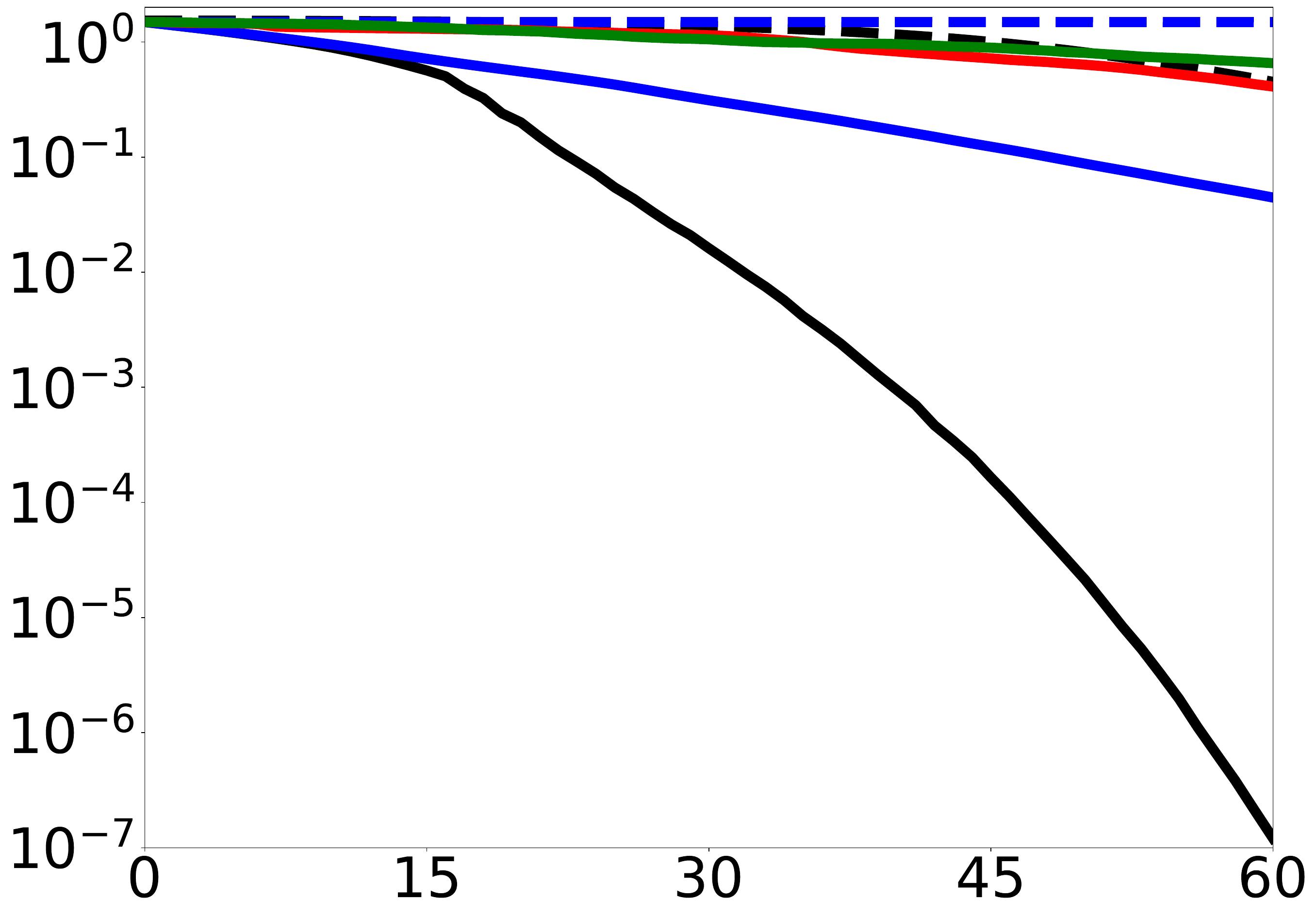}
         \caption{$r=60,\, p=40$ and $G=0$}
         \label{fig: Elip 2 mas SM}
     \end{subfigure}
        \hfill
    \begin{subfigure}[b]{0.325\textwidth}
         \centering
         \includegraphics[width=\textwidth]{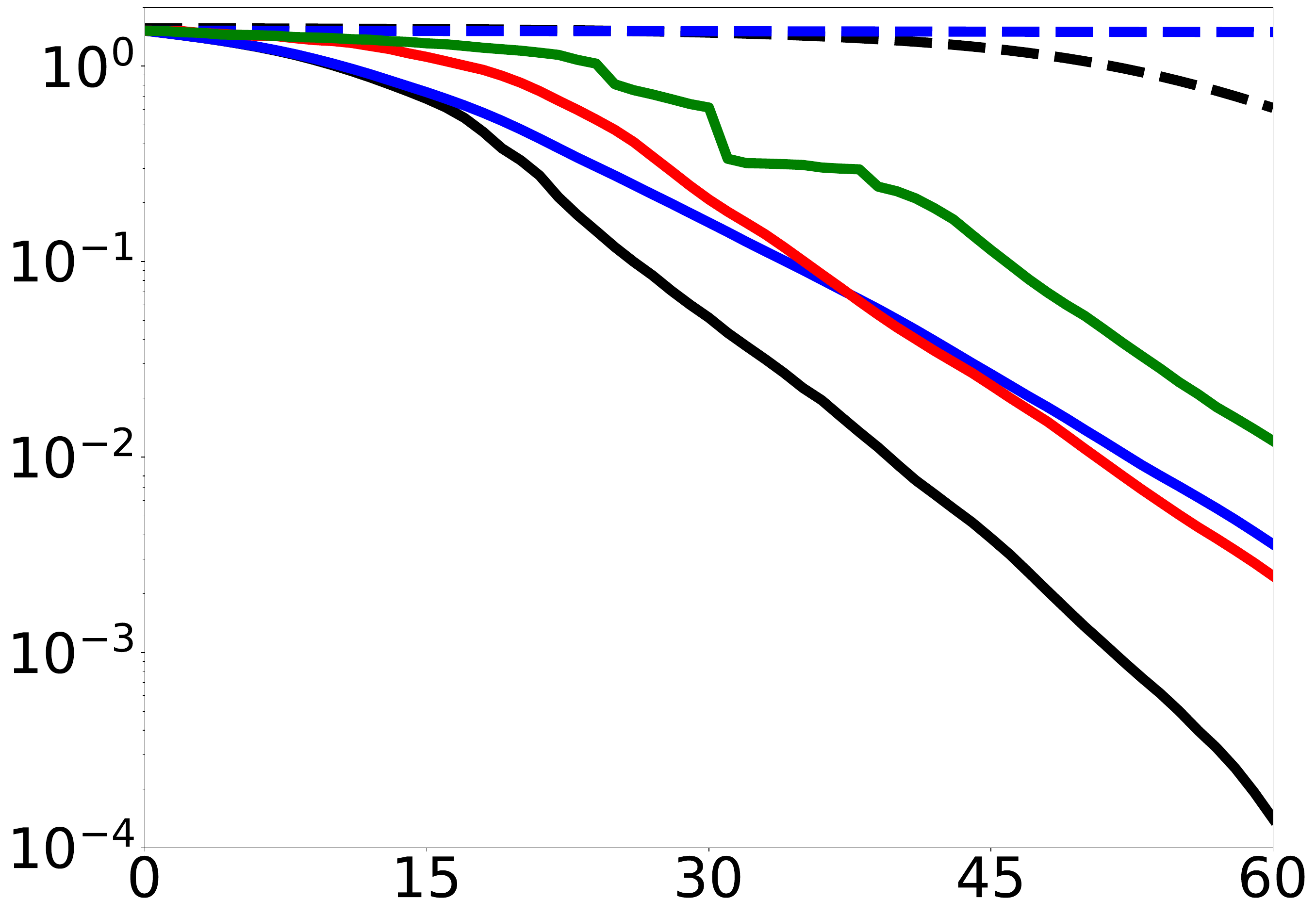}
         \caption{$r=30,\,p=27$ and $G=0.01$}
         \label{fig: Elip 2 poco SM}
     \end{subfigure}
        \hfill
    \begin{subfigure}[b]{0.325\textwidth}
         \centering
         \includegraphics[width=\textwidth]{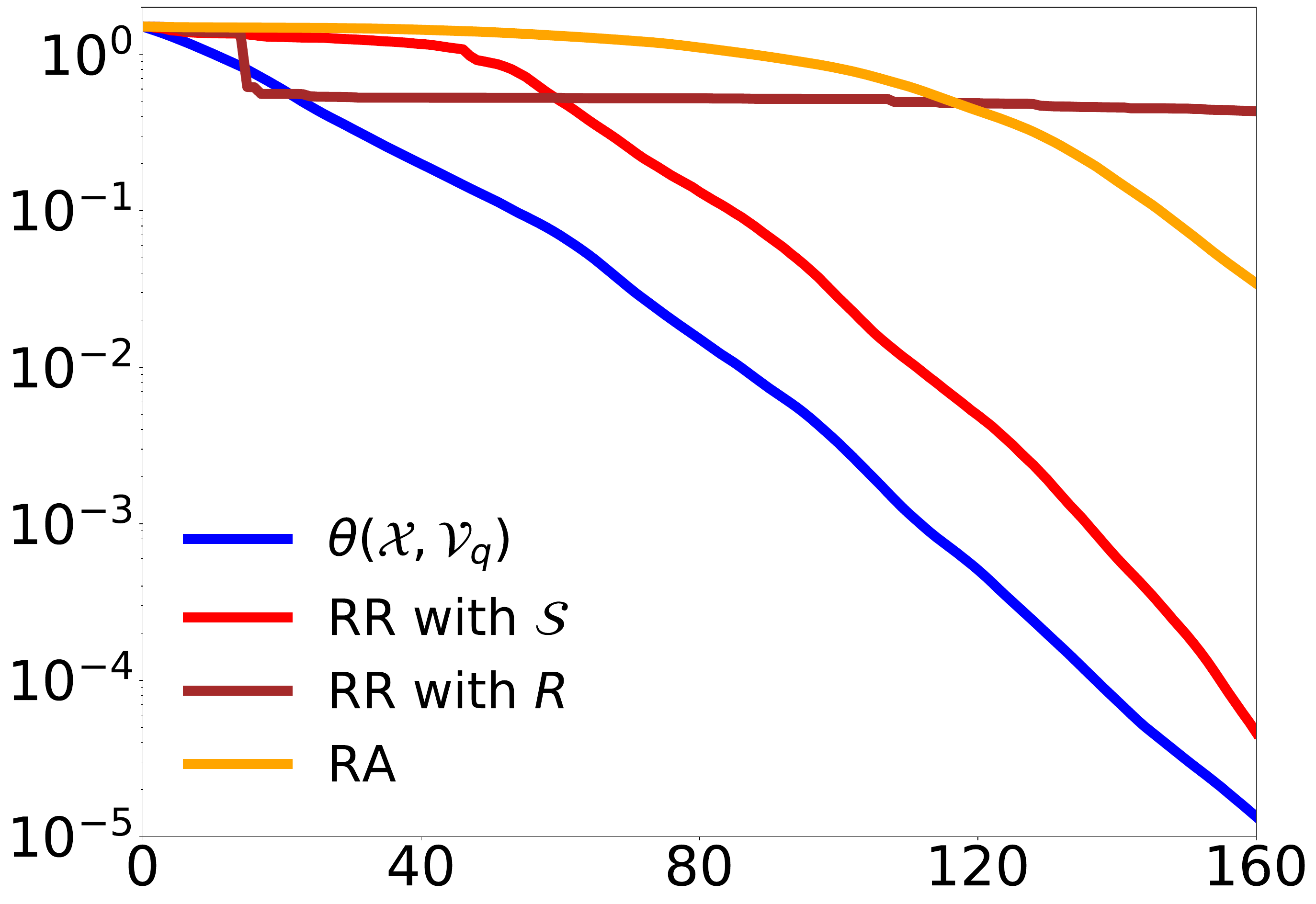}
         \caption{$r=45,\, p=30$ and $G=0$}
         \label{fig: eliptico rapido}
     \end{subfigure}
    \caption{Eliptical decay models}
    \label{fig: modelos elipticos}
\end{figure}

\begin{figure}[ht!]
     \centering
     \begin{subfigure}[b]{0.325\textwidth}
         \centering
         \includegraphics[width=\textwidth]{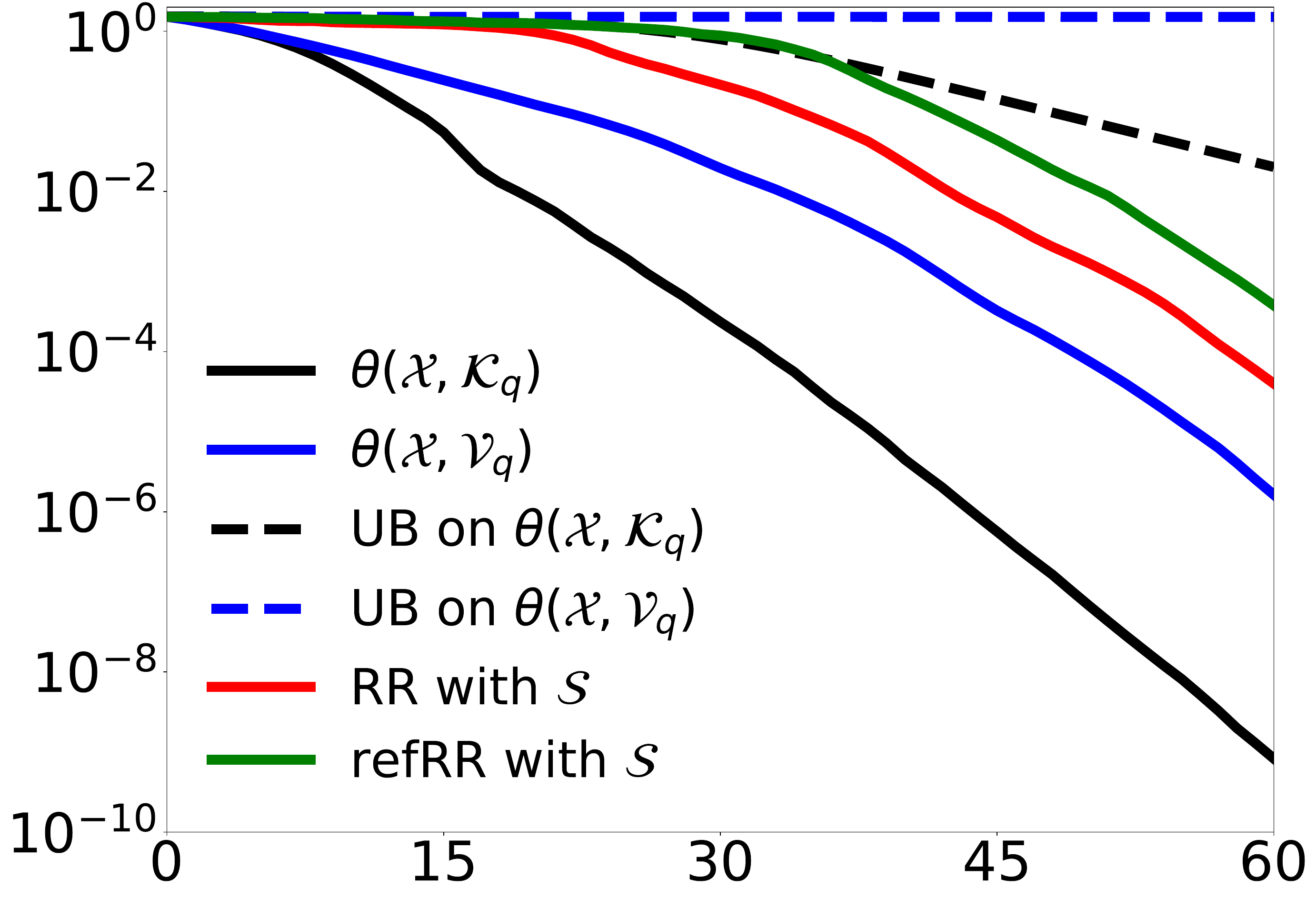}
         \caption{$r=30,\,p=20$ and $G=0$}
         \label{fig: poli plano}
     \end{subfigure}
     \hfill
     \begin{subfigure}[b]{0.325\textwidth}
         \centering
         \includegraphics[width=\textwidth]{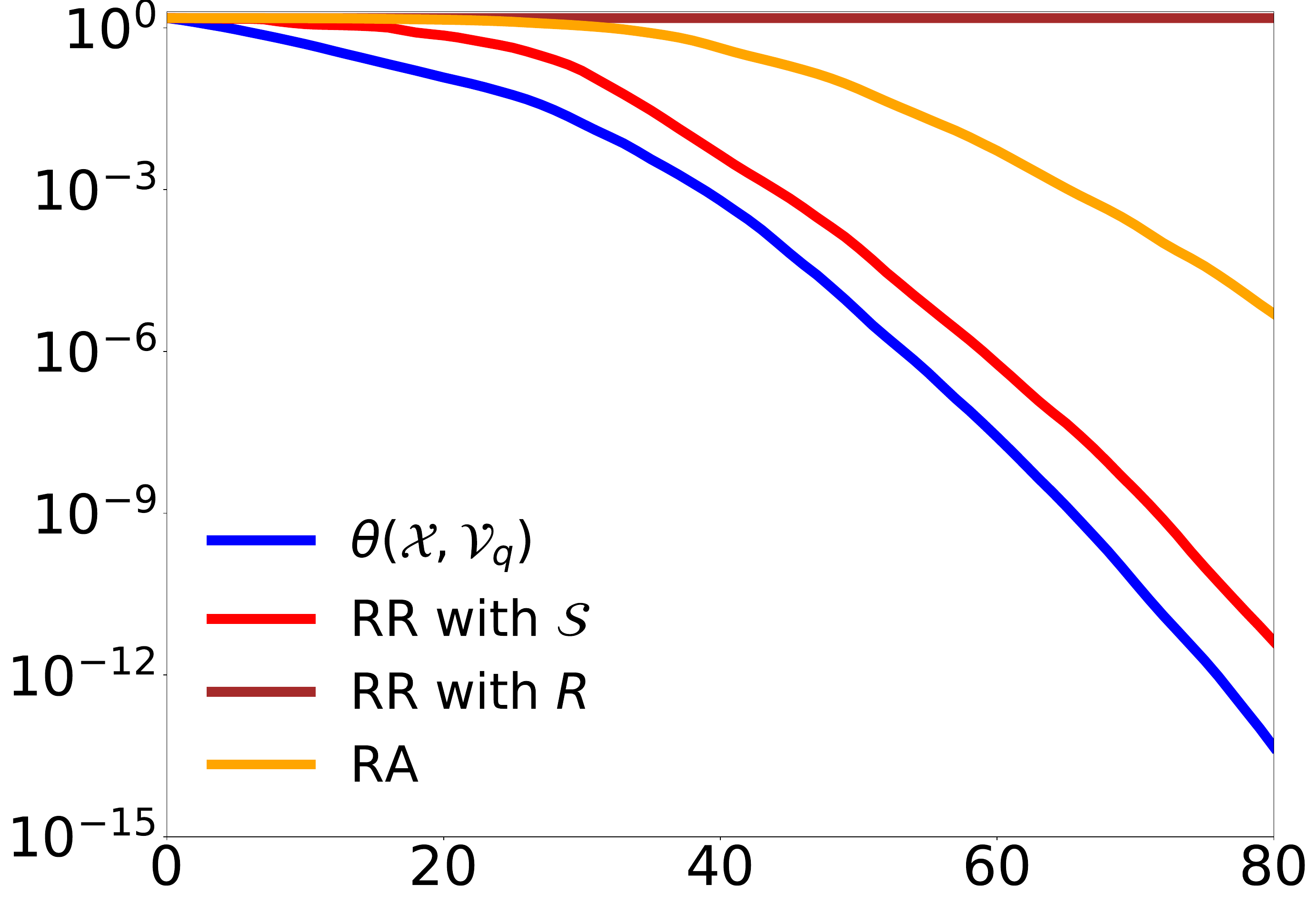}
         \caption{$r=30,\,p=20$ and $G=0$}
         \label{fig: poli rapidos}
     \end{subfigure}
     \hfill
     \begin{subfigure}[b]{0.325\textwidth}
         \centering
         \includegraphics[width=\textwidth]{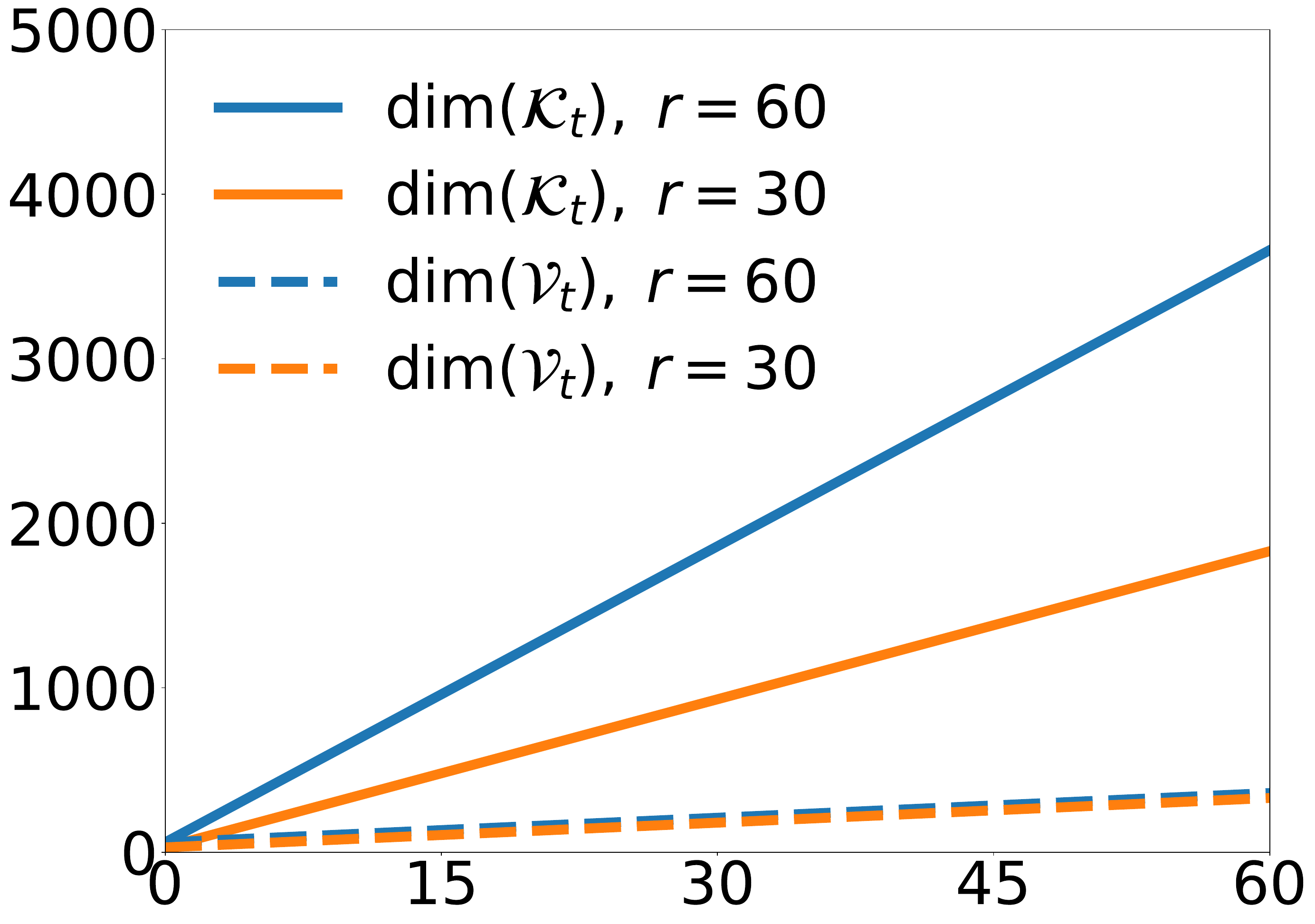}
         \caption{Dimensions of various subspaces}
     \end{subfigure}
        \caption{Polynomial (baseline) decay model with moderate oversampling and various dimensions.}
    \label{fig: modelos polinomicos}
\end{figure}

\pausa
{\bf Analysis of the numerical results}.
In all the computed examples, the block Krylov method produces the best 
approximation of the target space $\cX$ after 60 iterations. This is in accordance with the
fact that $\cV_t\subset \cK_t$, for $0\leq t\leq q$, and that $\dim \cV\geq \dim \cX$
(see the comments after Theorem \ref{teo main conv vt}).
On the other hand, the theoretical optimal subspace expansion computed
by Algorithm \ref{algo1} typically provides better approximations than its 
computable counterparts obtained using Algorithm \ref{algo2}, based on the
RR and ref-RR methods. 

\pausa
The approximations of the target subspace based on the implementation of 
Algorithm \ref{algo2} using the RR method typically have a better performance than
those approximations obtained using the ref-RR method. Although the ref-RR method
produces approximated eigenvectors that have smaller residual errors than those approximated eigenvectors obtained using the RR method \cite{J97, Jia04}, this fact does not seem to imply that the subspaces computed using Algorithm \ref{algo2} based on the ref-RR method have smaller angular distance that those computed using RR method.
Furthermore, RR projection methods are considerably faster than ref-RR projection methods (for the same number of iterations).
For these reasons we have not included the extension of Jia's expansion method based on the refined Rayleigh-Ritz, but we point out that, for $q=60$, our experiments with the extension of Jia's expansion method gave similar results for both projections methods.

\pausa
The upper bound in Theorem \ref{teo main Kt}
shows that both the decay of the
eigenvalues values of the matrix $A$ and the oversampling ($\rho=r-d$) of the method 
play a role in the decay of the (largest) angle $\theta(\cX,\cK_t)$ as a function of the 
number of iterations $t\geq 0$ (see the comments after Theorem \ref{teo main Kt} and in Remark \ref{rem sobre p}).  These facts are reflected in 
the numerical examples: the faster the eigenvalues values decay or the 
larger the oversampling the faster the angles $\theta(\cX,\cK_t)$ decay, 
as a function of $t$. Figures \ref{fig: lineal plano} and \ref{fig: Elip_1_plano} should be compared with Figures \ref{fig: lineal plano mucho OV} and \ref{fig: Elip 2 mas SM} respectively to get a grasp of the effects of oversampling in our experiments, keeping the same singular gap ($G=0$). 

\pausa
On the other hand, the upper bound in Theorem \ref{teo main conv vt}
shows that the gap between the eigenvalues $\la_d>\la_{d+1}$ 
of the matrix $A$ plays a role in the decay of the (largest) angle $\theta(\cX,\cV_t)$ as a function of $t\geq 0$.
Notice that the previous analysis does not involve the oversampling of the method.
The numerical examples confirm the role of the gap $\la_d>\la_{d+1}$ and
also show a relatively small impact of the oversampling of the method in the 
decay of the angle $\theta(\cX,\cV_t)$, as a function of $0\leq t\leq q$.  

\pausa
Regarding the growth of the dimensions of the computed subspaces, we point out that in 
all the numerical examples we get that $\dim \cK_t=(t+1)\cdot r$, while
$\dim\cV_t=\dim \widehat \cV_t=r+t\cdot d$, for $0\leq t\leq q$. In particular, 
after $q=60$ iterations with $r=30$ ($d=5$ and $n=5000$) we have $\dim \cK_{60}=1830$ ($\approx 36\%$ of 
$n=5000$) while $\dim \cV_{60}=\dim \widehat \cV_{60}=330$ ( $\approx 6\%$ of 
$n=5000$ and $\approx 18\%$ of $\dim \cK_{60}$). These last facts are also reflected in the computational times of the different methods. Indeed, Algorithm \ref{algo2} using the RR projection method becomes considerable faster than the block Krylov method as the number of iterations increases. For example, the computational time of block Krylov method for the elliptic model shown in Figure \ref{fig: Elip_1_plano} (with $r=45$) is comparable with the computational time of the Algorithm \ref{algo2} using the RR projection method shown in Figure \ref{fig: eliptico rapido}. Moreover, $\theta(\cX,\,\cK_{60})\approx\theta(\cX,\,\widehat{\cV}_{160})$ while $\dim\cK_{60}\approx2700$ and $\dim\widehat{\cV}_{160}\approx 850$.
Similar comments apply to the linear model (Figures \ref{fig: lineal plano} and \ref{fig: lineal rapidos}) and the polynomial model (Figures \ref{fig: poli plano} and \ref{fig: poli rapidos}).

\pausa
The performances of the theoretical optimal expansion $\cV_t$ as well as its
computable counterparts $\widehat \cV_t$ are comparable with the block Krylov method (for the same number of iterations)
in cases in which the eigenvalue gap $\la_d>\la_{d+1}$ is significant and the
oversampling of the methods is rather small. In the example shown in 
Figure \ref{fig: lineal plano poco OS y gap}, 
after $q=60$ iterations we have that $\dim \cK_{60}=1220$, 
while $\dim \cV_{60}=\dim \cV_{60}=320$ ($\approx 26\%$ of $\dim \cK_{60}$).

\section{Appendix: some technical results}\label{sec append complet}

In this section we present detailed proofs of the technical statements included in Section \ref{sec proofs main}. 

\subsection{Principal angles between subspaces}\label{Appendixity1}

\pausa
Below we present the proofs of Propositions \ref{pro monotonia de angulos} and \ref{rem desc sum y algo mas}. Although these results are arguably folklore, we include their proofs for the convenience of the reader.

\proof[Proof of Proposition \ref{pro monotonia de angulos}]
We consider subspaces  $\cX,\,\cY,\,\cS\subset \C^n$,
such that $d=\dim\cX\leq\dim\cY$ and $\cY\inc\cS$. 

\pausa
 To show that $\Theta(\cX,\,\cS)\leq  \Theta(\cX,\,\cY)$ notice that $\cY\subseteq\cS$, so we have that $P_\cS P_\cY=P_\cY=P_\cY P_\cS$. Then,  
		for every $ i\in \I_d\,$,  
\beq\label{inter}\barr{rl}
    \cos^2(\theta_i(\cX,\,\cY))
     =  s_i(P_\cX P_\cY)^2  & =     \la_i(P_\cY P_\cX  P_\cY) =
    \la_i(P_\cY (P_\cS P_\cX P_\cS) P_\cY)
		\\&\\
&   \stackrel{\star}\leq  \la_i(P_\cS P_\cX P_\cS)  =   s_i(P_\cX P_\cS)^2  = \cos^2(\theta_i(\cX,\,\cS))   \ ,
    \earr
   \eeq
where the inequalities $ \stackrel{\star}\leq$ follow by the interlacing inequalities. 
Thus, since the principal angles lay in $[0,\pi/2]$, the result follows from the fact that $\cos(x)$ is a decreasing function in that interval. 

\pausa 
Denote by  $\cM = P_\cS (\cX)\inc \cS $ and $m = \dim \cM \le d$. 
Notice that since $P_\cM P_\cX= P_\cS P_\cX $ then
$0<s_i(P_\cM P_\cX)=s_i (P_\cS P_\cX ) =\cos(\theta_i)$, for $1\leq i\leq m$.
Furthermore, $s_i (P_\cS P_\cX ) =s_i (P_\cM P_\cX ) =0$ for $m<i\le d$. These facts show 
that $\theta_i(\cX,\cM)\in [0,\pi/2)$ for $i\in\I_m$ and $\theta_i(\cX,\cM)=\pi/2$ for $m+1\leq i\leq d$, which is item 2.

\pausa 
 By the case of equality in the interlacing inequalities in Eq. \eqref{inter}, we get that $\Theta(\cX,\cS)= \Theta(\cX,\cY)$
(that is $\la_i(P_\cY  P_\cX P_\cY)=
\la_i( P_\cS P_\cX  P_\cS )$, for $1\leq i\leq d$) if and only if 
$\cY$ contains an orthonormal system $\{y_1,\ldots,y_d\}$ such that 
$$
P_\cS P_\cX  P_\cS \ y_i=\cos^2(\theta_i(\cX,\cS)\,)\  y_i \peso{for every} 
i\in \I_d\ .
$$ 
Since $\cos(\theta_i(\cX,\cS)\,)>0$ for $i \in \I_m\,$, this means that 
$\{y_1,\ldots,y_m\}  \inc P_\cS(\cX) = \cM$ (and they form an orthonormal basis for $\cM$). 
Hence $\Theta(\cX,\cS)= \Theta(\cX,\cY)$ if and only if $\cM\inc \cY$ which is item 3.
\QED

\proof [Proof of Proposition \ref{rem desc sum y algo mas}]
Let $\cV \coma \cA \inc \C^n$ be subspaces and set  $\cS = \cV+ \cA$.
To prove Eq. \eqref{eq desc sum1} i.e.
$\cS = \cV+ \cA=\cV\oplus (1-P_\cV)(\cA)$, let $v+a\in \cV+\cA $, for some $v\in\cV$ and $a\in\cA$. Then 
$$
v+a=P_\cV(v+a)+(1-P_\cV)(v+a)= v'+(1-P_\cV)a\in\cV\oplus(1-P_\cV)(\cA) \ , 
$$
where $v'=P_\cV(v+a)$ and where we used that $(1-P_\cV)v=0$. Conversely, if $v\in\cV$ and $a\in\cA$
then $v+(1-P_\cV)a=v+(a-P_\cV a)=v'+a$, where $v'=v-P_\cV a\in \cV$. 

\pausa 
\noindent Given a subspace $\cE$ such that $\cV\inc \cE \inc \cS$, simply take
 $\cD= \cA\cap \cE \subseteq \cA$, which
clearly satisfies that $\cV+\cD\inc \cE$. Moreover, given $e\in\cE\subseteq\cS$ there exists $a\in \cA$ and $v\in\cV$ such that $a=e-v \in \cD$ so $e=v+a\in\cV+\cD$. We therefore conclude that $\cE=\cV+\cD\stackrel{\eqref{eq desc sum1}}=\cV+(1-P_{\cV})\cD$.
\QED

\subsection{Principal angles and their tangents}\label{subsec angles}

    In this section we prove a key inequality involving tangents of principal angles between subspaces, that is related with the work in \cite{KZ13}. Indeed, if  
	we let	$\begin{bmatrix} X&X_\bot\end{bmatrix}$ be a unitary matrix with $X\in\C^{n\times d}$ and $V\in\C^{n\times r}$ with orthonormal columns and such that $\rk(X^*V)=d$ then, the first $d$ singular values of $X_\bot^*V(X^*V)^\dagger$ coincide with the  tangents of the angles between $R(X)$ and $R(V)$.
	In case that $V$ does not have orthonormal columns then the situation changes, as shown in the following

    \begin{exa}
Let 
$\begin{bmatrix} X&X_\bot\end{bmatrix}$ be the identity matrix of size 5, with $X\in\C^{5\times 2}$, and $\{V_a\}_{a\in\R}$ be the family of matrices given by
    \[
    V_a
    =
    \frac{1}{\sqrt{2}}
    \begin{pmatrix}
        1&0&0\\0&1&0\\1&0&0\\0&1&0\\a&0&\sqrt{2}
    \end{pmatrix}\in \C^{5\times 3}
    \,.
    \]
    It's easy to see that $\rk(V_a)=3$, $\rk(X^*V_a)=2$ and $R(V_a)=R(V_0)$ for all $a\in\R$, but
    
    \[
    X_{\bot}^*V_a(X^*V_a)^\dagger
    =
    \begin{pmatrix}
        1&0&0\\0&1&0\\a&0&\sqrt{2}
    \end{pmatrix}
    \begin{pmatrix}
        1&0&0\\0&1&0
    \end{pmatrix}^\dagger
    =
    \begin{pmatrix}
        1&0\\0&1\\a&0
    \end{pmatrix}
    \Longrightarrow
    s( X_{\bot}^*V_a(X^*V_a)^\dagger)
    =
    ((1+a^2)^{1/2},1)
    \,.
    \]
    \noindent So, the matrices $V_a$ are all full rank and they all have the same range, but the matrices $X_{\bot}^*V_a(X^*V_a)^\dagger$ may have an arbitrarily large first singular value, which doesn't contradict \cite[Theorem 3.1]{KZ13}, since $\rk(V_a)=3>2=\rk(X^*V_a)$. Notice that the choice $a=0$ produces the smallest singular values, which is exactly the case where the matrix $V_a$ has orthonormal columns. In that case, those singular values coincide with the tangents of the angles $\theta(X,\,V)$. \EOE
\end{exa} 

\pausa
The next result, that complements the result from Zhu and Knyazev, will play a key role
in our approach to the convergence analysis of Algorithm \ref{algo1} in the Hermitian case.
Notice that the next result is stronger than Theorem \ref{cor: tang vs raro}.

\begin{teo}\label{teo: tang vs raro}
Let $\begin{bmatrix} X&X_\bot\end{bmatrix}$ be a unitary matrix with $X\in\C^{n\times d}$ and let $V\in \C^{n\times r'}$ be such that $X^* V\neq 0$. Let $V_0\in \C^{n\times r}$
have orthonormal columns and be such that $R(V)=R(V_0)=\cV$. Then, we have
that $$s_j(X_\bot^* V_0(X^*V_0)^\dagger)\leq s_j(X_\bot^* V(X^*V)^\dagger) \ , \ \text{ for} \  1\leq j\leq d\,.$$
If we assume further that 
$\dim X^*(\cV)=d$ (so $r\geq d$) then, 
\begin{equation}\label{eq val sing}
s_j(\tan(\Theta(\cX,\, \cV)))
\leq
s_j(X_\bot^* V(X^*V)^\dagger) \ \text{ for}\ 1\leq j\leq d\,,
\end{equation}
\noindent where $\cX=R(X)$. In particular, for every unitary invariant norm $\|\cdot\|$ we have that
$$
\|X_{\bot}^* V_0(X^*V_0)^\dagger\|
=
\|\tan(\Theta(\cX,\cV))\|
\leq
\|X_{\bot}^* V(X^*V)^\dagger\|\,.
$$
\noindent 
\end{teo}

\begin{proof}
We first show that we can reduce our analysis to the case in which $\cX\cap \cV^\perp=\{0\}$ or, equivalently, that $\dim X^*(\cV)=d$. Indeed, let
$V\in \C^{n\times r'}$ be such that $\cV=R(V)$. Let $\cX\cap \cV^\perp=\cX_1$ and notice that by hypothesis
$\dim \cX_1<d$. Assume further that $1\leq \dim \cX_1=d_1$ and set $\cX_2=\cX\ominus \cX_1$. Since $\rk(X^*V)=d-d_1$, we have that $s_j(X_\bot^* V(X^*V)^\dagger)=0$ for $j>d-d_1$. Replacing $[X\ \ X_\perp]$ by $[X \ \ X_\perp ]=[XW \ \ X_\perp]$ for 
a convenient unitary matrix $W\in\C^{d\times d}$ we can further assume that $X=[X_1\ \ X_2]$, where 
$X_1$ and $X_2$ have columns that form orthonormal basis of $\cX_1$ and $\cX_2$, respectively. We remark that this replacement does not modify $s(X_\bot^* V(X^*V)^\dagger)$.
Then,
$$
X^*V=\begin{bmatrix} X_1^* \\ X_2^*\end{bmatrix}V=\begin{bmatrix} 0 \\ X_2^*V\end{bmatrix} \implies
(X^*V)^\dagger =[0 \ \ (X_2V)^\dagger ]\,.
$$Hence, $X_\perp^*V (X^*V)^\dagger= X_\perp^*V  [0 \ \ (X_2V)^\dagger ]=[ 0 \ \ X_\perp^*V (X_2V)^\dagger]$.
Notice that the columns of the matrix $\tilde U=[X_2 \ \ X_\perp]\in\C^{n\times (n-d_1)}$ form an orthonormal basis of $\C^n \ominus \cX_1=\cX_2 \oplus \cX^\perp$. Consider $I_{n-d_1}=\tilde U^* \tilde U=[H_1 \ \ H_{1,\perp}]\in \C^{(n-d_1)\times (n-d_1) }$, where the columns of $H_1$ and $H_{1,\perp}$ are the first $d-d_1$ and last $n-d$ vectors of the canonical basis of $\C^{n-d_1}$ and set $\tilde V:= \tilde U^* V\in \C^{(n-d_1)\times r'}$. By construction we have that 
$X_1^*V=0$ and $\tilde U\,{\tilde U}^*=I_n-X_1X_1^*$ and thus, $\tilde U\,\tilde V=(I_n-X_1X_1^*)V=V$.
Since $\tilde U$ has orthonormal columns
$$\tilde U (R(\tilde V)^\perp)= [\tilde U (R(\tilde V))]^\perp \ominus \cX_1 = \cV^\perp \ominus \cX_1\,.$$
We also have that $$\tilde U(R(H_1)\cap R(\tilde V)^\perp)= \tilde U(R(H_1)) \cap \tilde U (R(\tilde V)^\perp)= 
\cX_2\cap (\cV^\perp\ominus \cX_1)=\{0\}\,.$$ Thus, $R(H_1)\cap R(\tilde V)^\perp=\{0\}$. 
Moreover, we get that
$H_{1,\perp}^* \tilde V (H_1^* \tilde V)^\dagger = X_\perp ^*V (X_2^*V )^\dagger$. The previous facts show that 
$$s_j( X_\perp^*V (X^*V)^\dagger) =s_j(H_{1,\perp}^* \tilde V (H_1^* \tilde V)^\dagger)\ , \ \text{ for}\ 1\leq j\leq d-d_1\,.$$

\pausa
Thus, we can replace the triplet $(X,X_\perp,V)$ by $(H_1,H_{1,\perp},\tilde V)$ in such a way that $R(H_1)\cap R(\tilde V)^\perp=\{0\}$. 
Notice that $\tilde U$, $H_1$ and $H_{1,\perp}$ only depend on $[X \ \ X_\perp]$ and $\cV$; also notice that $\tilde V=\tilde U^* V$ has orthonormal columns whenever $V$ has orthonormal columns.  Thus, in what follows we assume further that $\cX\cap \cV^\perp=\{0\}$ or equivalently, that $\dim X^*(\cV)=d$.

\pausa
Under this assumption, 
it follows from \cite[Theorem 3.1]{KZ13} that $s_j(X_\perp ^* V_0(X^*V_0)^\dagger )$ is equal to $s_j(\tan (\Theta(\cX,\,\cV) ))$, for $1\leq j\leq d$. Thus, if we prove the singular value inequalities in Eq. \eqref{eq val sing} then the first set of inequalities will follow. Also, by means of Ky Fan's dominance theorem (see \cite[Theorem IV.2.2]{Bhatia}) Eq. \eqref{eq val sing} also implies the last inequality in the statement, so we are only left to prove Eq. \eqref{eq val sing}

\pausa
By the previous argument we now consider a unitary matrix $[X\ \ X_\perp]\in\C^{n\times n}$ and $V\in \C^{n\times r'}$ such that if $\cV=R(V)$ and $\cX=R(X)$ then $\cX\cap \cV^\perp=\{0\}$.
Next we show that we can further assume that $V$ has full column rank. Indeed, we can always choose an orthonormal basis $\{w_1,\ldots,w_r\}\subseteq\C^{r'}$ for $\ker(V)^\perp$ and another orthonormal basis $\{w_{r+1},\ldots,w_{r'}\}\subseteq\C^{r'}$ for $\ker(V)$ to produce a unitary matrix $W=[w_{1} \cdots w_{r'}]\in\C^{r'\times r'}$ and consider $VW$. Notice that $VW=\begin{bmatrix} V_1&0\end{bmatrix}$ where $V_1\in \C^{n\times r}$ has full column rank and the same range as $V$, so the LHS of Eq. \eqref{eq val sing} remains the same when we replace $V$ with $V_1$. We now need to show that this is also the case for the RHS of Eq. \eqref{eq val sing}, which is done directly:
\begin{align*}
    X_\bot^* V(X^*V)^\dagger
    &		=
    X_\bot^* VW\;W^*(X^*V)^\dagger
		=
    X_\bot^* VW\;(X^*VW)^\dagger
    \\
    &=
    X_\bot^* \begin{bmatrix} V_1&0\end{bmatrix} \; (X^*\begin{bmatrix} V_1&0\end{bmatrix})^\dagger
		= 
    \begin{bmatrix} X_\bot^* V_1&0\end{bmatrix} \; \begin{bmatrix} X^*V_1&0\end{bmatrix}^\dagger
    \\
    &= 
    \begin{bmatrix} X_\bot^* V_1&0\end{bmatrix} \; \begin{bmatrix} (X^*V_1)^\dagger\\0\end{bmatrix}
    = 
    X_\bot^* V_1(X^*V_1)^\dagger
\end{align*}
\noindent since $W$ is unitary. Thus, we further assume (for now) that $V=V_1$ is full column rank and that $r'=r$.

\pausa
Let $\cS=\ker(X^*V)$: our assumptions imply that $\cS\subseteq\C^r$, $\dim\cS=r-d$ and that $\dim V(\cS^\bot)=\dim \cS^\perp =d$, where we used that $V$ is full column rank. 
Hence, by item 1. in Proposition \ref{pro monotonia de angulos}, we have that $\Theta(X,\,V)\leq\Theta(X,\,V(\cS^\bot))$. 
This implies that, for $1\leq j\leq d$
\begin{equation}\label{eq reduccion 1}
    s_j(\tan(\Theta(X,\,V)))
    \leq
    s_j(\tan(\Theta(X,\,V(\cS^\bot))))
    \,.
\end{equation}

\noindent If we let $\tilde{V}=VP_{\cS^\bot}$, then $R(\tilde V)=V(\cS^\perp)$, so $\rk(\tilde{V})=d$, and $X^*\tilde{V}=X^*VP_{\cS^\bot}=X^*V$, so 
$$\ker(X^*\Tilde{V})=\ker(X^*V)=\cS=\ker(\Tilde{V})\,.$$ Using the fact that $R((X^*V)^\dagger)=\ker(X^*V)^\bot=\cS^\bot$ we have that
\begin{equation}\label{eq reduccion 2}
X_\bot^* \tilde{V} (X^*\tilde{V})^\dagger
=
X_\bot^* VP_{\cS^\bot} (X^*VP_{\cS^\bot})^\dagger
=
X_\bot^* V(X^*V)^\dagger\,,
\end{equation}
\noindent and combining \eqref{eq reduccion 1} with \eqref{eq reduccion 2} it is easy to see that if we can prove that \eqref{eq val sing} holds for $\Tilde{V}$ then it also holds for $V$. To save notation, we will write $V$ instead of $\tilde{V}$, dropping the assumption that $V$ has full column rank and assuming that it satisfies the following conditions: $\ker (X^*V) = \ker V$ and $\rk(V)=\rk(X^*V)=d$.

\pausa
Let us take the polar decomposition $V=U|V|$ of $V$; in this case $U\in \C^{n\times r}$ is a partial isometry (i.e. $UU^*=P_{R(U)}$ and $U^*U=P_{\ker(U) ^\perp }$) such that $R(U)=R(V)$ and $\ker U = \ker V = \ker(X^*V)$. Since $R(U)=R(V)$ we have that $R(X^*U)=R(X^*V)$, so $\rk(X^*U)=d$. Moreover, $\ker(X^*U)=\ker(U)$: Indeed, the inclusion $\ker(U)\subset \ker(X^*U)$ is trivial and the dimensions coincide:
\[
\dim\ker(X^*U)
=
r-\rk(X^*U)
=
r-\rk(X^*V)
=
r-\rk(V)
=
r-\rk(U)
=
\dim\ker(U)\,.
\]

\pausa
Now, consider a SVD of $X^*U$ given by $X^*U=Z\Sigma W^*$, where $Z\in \C^{d \times d}$ and $W\in \C^{ r \times r}$ are unitary matrices.  Let $\{z_1,\ldots,z_d\}$ and $\{w_1,\ldots,w_d\}$ denote the first $d$ columns of $Z$ and $W$ respectively. Notice that $\{w_1,\ldots,w_d\}$ is a basis for the subspace $\ker(X^*U)^\perp=\ker(U)^\bot = \ker(V)^\bot =\ker(X^*V)^\bot$ and
\[
X^*U w_i=\cos(\theta_i) \ z_i \ , \ 1\leq i\leq d
\]
\noindent where $\theta_i = \theta_i(X,\,V)$, for $1\leq i\leq d$. Taking the adjoint of the previous SVD we get $U^*X=W\Sigma Z^*$ so
\[
U^*X z_i=\cos(\theta_i) \ w_i \ , \ 1\leq i\leq d\,.
\]

\pausa
Using the fact that $R(|V|^\dagger )=R(|V|)=\ker V^\perp=\ker U^\perp=(\ker X^*V)^\perp$ we get that
\[
\cos(\theta_i)\ z_i
=
X^*U |V| \ |V|^\dagger w_i
=
X^*V (|V|^\dagger w_i)
\Rightarrow
\cos(\theta_i)(X^*V)^\dagger z_i
=
|V|^\dagger w_i
\]
\noindent and it follows that
\[
X_\bot^* V(X^*V)^\dagger z_i
=
\frac{1}{\cos(\theta_i)} X_\bot^*U|V| |V|^\dagger w_i
=
\frac{1}{\cos(\theta_i)} X_\bot^*U w_i
\ , \ 1\leq i\leq d\,.
\]

\noindent Next, notice that $\{w_i\}_{1\leq i\leq d}\subseteq\ker U^\bot$, where $\begin{bmatrix} X&X_\bot \end{bmatrix}^*U$ acts isometrically. So, the vectors
\[
\begin{bmatrix} X&X_\bot \end{bmatrix}^*Uw_i
=
\begin{bmatrix} X^*Uw_i\\ X_\bot^*Uw_i \end{bmatrix}
=
\begin{bmatrix} \cos(\theta_i)z_i\\ X_\bot^*Uw_i \end{bmatrix}
\text{ for } 1\leq i\leq d
\]
\noindent form an orthonormal system. From this (and the fact $\{z_i\}_{1\leq i\leq d}$ is an orthonormal system) we deduce that $\{X_\bot^*Uw_i\}_{1\leq i\leq d}$ is an orthogonal system with $\|X_\bot^*Uw_i\|=\sin(\theta_i)$, for $1\leq i\leq d$.

\pausa
We now consider the following variational characterization of the singular values (see \cite[Problem III.6.1]{Bhatia}): if we let $\cM_j$ be the $j$-dimensional subspace spanned by $\{z_i\, : \, d-j+1\leq i\leq d\}$ then, for $1\leq j\leq d$ we have that
\begin{align*}
    s_j^2(X_\bot^* V(X^*V)^\dagger)
    &\geq
    \min_{x\in \cM_j,\;\|x\|=1} \|X_\bot^* V(X^*V)^\dagger x\|^2
    \\
    &=
    \min_{\sum_{i=d-j+1}^d |\alpha_i|^2=1} \left\|\sum_{i=d-j+1}^d \frac{\alpha_i}{\cos(\theta_i)} X_\bot^* U w_i\right\|^2
    \\
    &=
    \min_{\sum_{i=d-j+1}^d |\alpha_i|^2=1} \sum_{i=d-j+1}^d \frac{|\alpha_i|^2}{\cos^2(\theta_i)} \|X_\bot^* U w_i\|^2
    \\
    &=
    \min_{\sum_{i=d-j+1}^d |\alpha_i|^2=1} \sum_{i=d-j+1}^d |\alpha_i|^2\tan^2(\theta_i)
    \\
    &\geq \tan^2(\theta_{d-j+1})
    =
    s_j^2(\tan(\Theta(X, V)))
    \,.
\end{align*}
\noindent Taking square roots on both sides, the proof is complete.

\end{proof}

\subsection{Analysis of iterative algorithms}\label{subsec algos}

\proof[Proof of Theorem \ref{teo cota con polinomio}]
We fix an Hermitian matrix $A\in\C^{n\times n}$ and consider
an eigen-decomposition $A=X\Lambda X^*$ , where the matrix $X\in \C^{n\times n}$ is unitary and 
$\Lambda=\text{diag}(\la_1,\ldots,\la_n)$, with $\la_1\geq \ldots\geq \la_n$. We let $x_1,\ldots,x_n$ denote the columns of $X$. Hence, $A\,x_j=\la_j\,x_j$, for $1\leq j\leq n$. We assume further that $\la_d>\la_{d+1}$. Furthermore, for $1\leq \ell\leq n$ we consider the partitions:
$$
X=\begin{bmatrix} X_\ell& X_{\ell,\bot} \end{bmatrix} \peso{and} \Lambda=\begin{bmatrix} \Lambda_\ell & 0 \\ 0 & \Lambda_{\ell,\bot} \end{bmatrix} \,.
$$ where $X_\ell\in \C^{n\times \ell}$ and $\Lambda_\ell\in\C^{\ell\times \ell}$. 

\pausa
We let $\cV\subset \C^n$ with $\dim\cV=r\geq d$ and 
	 fix $d\leq p\leq r$ such that $\text{span}\{x_1,\ldots,x_p\}\cap \cV^\perp=\{0\}$. 
In what follows	we find $\cH_p\subset R(V)$ with $\dim\cH_p=d+r-p$
		such that
		$$    \|\tan(\Theta(\cX,\, \phi(A)(\cV) ))\|
        \leq
         \|\phi(\Lambda_d)^{-1}\|_2\ \|\phi(\Lambda_{p,\perp})\|_2 
        \;\| \tan(\Theta(\cX,\, \cH_p))\|
    $$
 for every polynomial $\phi\in \C[x]$ such that $\phi(\Lambda_d)$ is invertible and for every u.i.n. $\|\cdot\|$.

\pausa
Indeed, let $W\in\C^{n\times r}$ be a matrix with orthonormal columns such that $R(W)=\cV$.
Since $\text{span}\{x_1,\ldots,x_p\}\cap \cV^\perp=\{0\}$ we get that $p=\rk(W^*X_p)=\rk(X_p^*W)$.
    We now consider the matrix
    $F=[x_{d+1} \cdots x_{p}]^*W\in \C^{(p-d)\times r}$. The fact that $\rk(X_p^*W)=p$ implies that $\rk(F)=p-d$ so $\dim\ker(F)=r-(p-d)$. Let $Z\in\C^{r\times (d+r-p)}$ be such that it has orthonormal columns and $R(Z)=\ker(F)$. Then, $WZ\in \C^{n\times (d+r-p)}$ has orthonormal columns and is such that $R(WZ)\subset \{x_{d+1},\ldots,x_p\}^\perp$. Also, $R(WZ)\subset R(W)=\cV$ and $\rk(WZ)=d+r-p\geq d$. We define $\cH_p=R(WZ)$ and check that this subspace has the desired properties. Indeed, the previous comments show that $\cH_p\subset \cV$ is such that $\dim \cH_p=d+r-p$.
    
\pausa
  Since $\dim \ker(X_p^*W)^\perp=p$ we see that $\dim R(Z)+\dim \ker(X_p^*W)^\perp=d+(r-p)+p=r+d$ so we must have 
    $\dim \ker(X_p^*W)^\perp\cap R(Z)\geq d$ so then $\rk(X_p^*WZ)\geq d$. But, by construction, 
$$
X_p^*WZ=[[x_1\cdots x_d] \ [x_{d+1}\cdots x_p]]^* WZ= \begin{bmatrix} X_d^*WZ \\ 0 \end{bmatrix}\,.
$$ Therefore, we conclude that $\rk(X_d^* WZ)=d$. Since $WZ$ has orthonormal columns, $R(WZ)=\cH_p$ and $\dim \cH_p\geq \dim \cX$, we conclude that $\cX\cap \cH_p^\perp=\{0\}$.

\pausa Using the previous facts we can see, since $\phi(\Lambda_d)$ is invertible, that $\rk(\phi(A)WZ)\geq d$. Also, $R(\phi(A)WZ) \subset R(\phi(A)W)=\phi(A)(\cV)$.
Thus, using item 1. in Proposition \ref{pro monotonia de angulos} and Theorem \ref{cor: tang vs raro} we have that ($R(X_d)=\cX$)
    \begin{align*}
        \|\tan(\Theta(\cX, \phi(A)(\cV) ))\|
				&= \|\tan(\Theta(\cX,\, \phi(A)W ))\|
				\\ 
				&\leq \|\tan(\Theta(\cX,\, \phi(A)W Z ))\|
				\\
				&\leq \|X_{d,\bot}^* \phi(A)WZ \; (X_d^*\phi(A)WZ)^\dagger\|
        \\
        &=
        \|X_{d,\bot}^* X\phi(\Lambda)X^*WZ \; (X_d^*X\phi(\Lambda)X^*WZ)^\dagger\|
        \,.
    \end{align*}
Now, we can simplify 
$$X_{d,\bot}^*X\phi(\Lambda)X^*WZ=\phi(\Lambda_{d,\bot})X_{d,\bot}^*WZ=MX_{d,\bot}^*WZ\,,$$
where $M=\begin{bmatrix}
    0_{p-d}&0\\0&\phi(\Lambda_{p,\bot})
\end{bmatrix}$ by using the fact that $X_{d,\bot}^*WZ=[0 \ X_{p,\bot}]^*WZ$ by construction of $Z$.
Similarly,
$$X_d^*X\phi(\Lambda)X^*WZ = 
\phi(\Lambda_d)X_d^*WZ\,.$$ 
 Using these identities in our previous estimation and using the fact that $\phi(\Lambda_d)$ is invertible and $X_d^*WZ$ has full row rank, we obtain
    \begin{align*}
        \|\tan(\Theta(\cX,\, \phi(A)(\cV) ))\|
        &\leq
        \|MX_{d,\bot}^*WZ(\phi(\Lambda_d)X_d^*WZ)^\dagger\|
        \\
        &=
        \|MX_{d,\bot}^*WZ (X_d^*WZ)^\dagger \phi(\Lambda_d)^{-1}\|
        \,.
    \end{align*}
    Finally, by the strong sub-multiplicativity of the unitarily invariant norms
				we can conclude that
    \begin{align*}
        \|\tan(\Theta(\cX,\, \phi(A)(\cV) ))\|
        &\leq
        \|MX_{d,\bot}^*WZ (X_d^*WZ)^\dagger \phi(\Lambda_d)^{-1}\|
        \\
        &\leq
        \|M\|_2 \; \|X_{d,\bot}^*WZ (X_d^*WZ)^\dagger\| \; \|\phi(\Lambda_d)^{-1}\|_2
        \\
        &= \|\phi(\Lambda_{p,\bot})\|_2 \; \|\tan(\Theta(\cX,\, \cH_p ))\| \; \|\phi(\Lambda_d)^{-1}\|_2
    \end{align*} where we have used 		
		Zhu and Knyazev's result from \cite{KZ13} in the last identity,
	since $WZ$ and $X_d$ have orthonormal columns, $R(WZ)=\cH_p$, $R(X_d)=\cX$ and $\rk(X_d^* WZ)=d$	
		(see the discussion at the beginning of Section \ref{subsec angles}). Hence, the subspace $\cH_p$ has the desired properties.
\QED

\pausa
The next technical lemma is needed for the proof of item $1.$ in Proposition \ref{pro estimac poly main}.

\begin{lem}\label{lem la tildef}
    Suppose that $\lambda_1\geq\ldots,\geq\lambda_d>\lambda_{d+1}\geq\ldots\geq\lambda_n= 0$ are non negative real numbers and define $f:\R\setminus\{\lambda_1,\ldots,\lambda_d\}\rightarrow\R $ as
    \[
    f(s) 
    =\left( \min_{1\leq j\leq d} |s-\lambda_j|\right)^{-1}\;
    \max_{d+1\leq j\leq n} |s-\lambda_j|        
    \,.
    \]
    Then, f achieves it's absolute minimum at $s=\lambda_{d+1}/2$ and that minimun value is less than the singular gap $\lambda_{d+1}/\lambda_d<1$.
\end{lem}

\begin{proof}
    Let $m=\lambda_{d+1}/2$. We first show that $f(m)$  is strictly less than the singular gap:
    \[
    f(m)
    =
\frac{1}{\lambda_d-\left(\frac{\lambda_{d+1}}{2}\right)}\;     \frac{\lambda_{d+1}}{2}
    =
    \frac{\lambda_{d+1}}{\lambda_d+(\lambda_d-\lambda_{d+1})}
    <
    \frac{\lambda_{d+1}}{\lambda_d}
    \]
    Now, to show that this is the minimum value of $f$, we consider the following cases:
    \begin{enumerate}
        \item If $s< m$ then, since $\lambda_{d+1}-\lambda_d<0$ we have that
        \[
        f(s)
        =
        \frac{\lambda_{d+1}-s}{\lambda_d-s}
        =
        1+\frac{\lambda_{d+1}-\lambda_d}{\lambda_d-s}
        >
        1+\frac{\lambda_{d+1}-\lambda_d}{\lambda_d-m}
        =
        \frac{\lambda_{d+1}-m}{\lambda_d-m}
        =
        f(m)\,.
        \]
        
        \item If $s\in (m,\,\lambda_d)$ then,
        \[
        f(s)
        =
        \frac{s}{\lambda_d-s}
        =
        -1+\frac{\lambda_d}{\lambda_d-s}
        >
        -1+\frac{\lambda_d}{\lambda_d-m}
        =
        \frac{m}{\lambda_d-m}
        =
        f(m)\,.
        \]
        
        \item  If $s>\lambda_d$ then, 
        \[
        f(s)
        =
        				\left( \min_{1\leq j\leq d} |s-\lambda_j|\right)^{-1}
				 s \;
                \geq
        \frac{s}{s-\lambda_d}
        \geq 1> \frac{\lambda_{d+1}}{\lambda_d}
        > f(m)\,.
        \]
    \end{enumerate}
\end{proof}

\proof[Proof of item $1.$ in Proposition \ref{pro estimac poly main}] Consider Notation \ref{notac2}. For $s\in \R$ we let
$\phi_s(x)=x-s$ and set
$f:\R\setminus\{\lambda_1,\ldots,\lambda_d\}\rightarrow\R $ as
    \[
    f(s)
    =
     \|\phi_s(\Lambda_d)^{-1}\|_2\ \|\phi_s(\Lambda_{d,\bot})\|_2
    \,.
    \]
In what follows we show that $f$ achieves it's global minimum on $(\lambda_{d+1}+\lambda_n)/2$.

\pausa
    Set $\tilde{\lambda_j}=\lambda_j-\lambda_n\geq 0$ for $1\leq j\leq n$ so that $s-\lambda_j=(s-\lambda_n)-(\lambda_j-\lambda_n) = (s-\lambda_n)-\tilde{\lambda}_j$. Then, 
    \[
    f(s)
    =
    \left( \min_{1\leq j\leq d} |s-\lambda_j|\right)^{-1}\;
		\max_{d+1\leq j\leq n} |s-\lambda_j| 
      =
       \left( \min_{1\leq j\leq d} |(s-\lambda_n)-\tilde{\lambda}_j|\right)^{-1}
			\;  \max_{d+1\leq j\leq n} |(s-\lambda_n)-\tilde{\lambda}_j| 
    \]
    \noindent so $f(s)=\tilde{f}(s-\lambda_n)$ where 
		$\tilde{f}(s)$ is the function considered in Lemma \ref{lem la tildef} for the real non negative numbers $\tilde{\lambda}_1\geq\ldots\geq\tilde{\lambda}_n=0$. Since $\tilde{f}$ achieves it's minimum at $\tilde{\lambda}_{d+1}/2$ we get that $f$ achieves it's minimum at $\tilde{\lambda}_{d+1}/2+\lambda_n=(\lambda_{d+1}+\lambda_n)/2$ as stated. Finally, we estimate
    \[
    f((\lambda_{d+1}+\lambda_n)/2)
    =
    \tilde{f}(\tilde{\lambda}_{d+1}/2)
    =
    \frac{\tilde{\lambda}_{d+1}}{\tilde{\lambda}_d+(\tilde{\lambda}_d-\tilde{\lambda}_{d+1})}
    =
    \frac{\lambda_{d+1}-\lambda_n}{\lambda_d+(\lambda_d-\lambda_{d+1})-\lambda_n}
    <1
    \,.
    \]
\QED

\subsection{Upper bounds for Chebyshev polynomials}\label{subsec cheby}

It is well known that Chebyshev polynomials play an important role in the analysis of block Krylov methods \cite{Drineas,MuscoMusco,WWZ}. Here we will review the definition and a few simple results regarding the Chebyshev polynomials of the first kind. These will allow us to show that they have some separation properties that we can exploit in our favor. Specifically, we will use these polynomials in the proof of Theorem \ref{teo main Kt}.

\begin{fed}
    The Chebyshev polynomial of the first kind of degree $t$, denoted by $T_t(x)\in \R[x]$, is the polynomial defined by the relation
    \begin{equation}\label{eq T_n def}
        T_t(x)
        =
        \cos(t\,\theta)
        \;\text{ when }\;
        x=\cos(\theta)
        \ , \text{ for }\;
        \theta\in [0,\,\pi]\,.
    \end{equation}
    The fact that this defines a polynomial in $x$ follows from the De Moivre's formula, since it proves that $\cos(t\,\theta)$ is a polynomial of degree $t$ in $\cos(\theta)$.
\end{fed}

\pausa
Using the trigonometric relation $\cos(t\theta) + \cos((t-2)\theta) = 2 \cos(\theta) \cos((t - 1)\theta)$ one can obtain the fundamental recurrence relation: $T_0(x)=1$, $T_1(x)=x$ and
\begin{equation}\label{eq T_n recurrencia}
    T_t(x) 
    =
    2\,x\,T_{t-1}(x) - T_{t-2}(x)
    \ , \text{ for } t\geq 2\,.
\end{equation}

\pausa
The following properties of the Chebyshev polynomials are shown in sections 1.4 and 2.4 of \cite{Cheby}. 

\begin{pro}
For $x\geq 1$ the Chebysehv polynomials $T_t(x)$ can be  presented as
\begin{equation}\label{eq T_n forma 3}
    T_t(x)
    =
    \frac{(x+\sqrt{x^2-1})^t+(x-\sqrt{x^2-1})^t}{2}\,.
\end{equation}
\noindent Also, for all $x$ we have the following expression for their derivatives:
\begin{equation}\label{eq T_n derivada}
    T_t'(x)
    =
    \begin{cases}
        2\,t\, (T_{t-1} + T_{t-3} + \ldots + T_{1}) &\text{if } t \text{ is even}\\
        2\,t\, (T_{t-1} + T_{t-3} + \ldots + T_{2})+t &\text{if } t \text{ is odd}
    \end{cases}
\end{equation}
\end{pro}

\pausa
Now we will establish a few simple properties of these polynomials.

\begin{pro}
    The polynomials $T_t(x)$ for $t\geq 0$ satisfy the following:
    \begin{enumerate}
        \item $T_t(1)=1$, $|T_t(x)|\leq 1$ when $|x|\leq 1$ and $T_t(x)\geq 0$ when $x\geq 1$.
        \item for $x\geq 1$ we have that $T_{t}(x)\leq 2\,x\,T_{t-1}(x)$ and $T_t'(x)\geq 2\,t\,T_{t-1}(x)\geq 0$. In particular, $T_t$ is monotonically increasing in $[1,\,\infty)$.
        \item for $z\geq y\geq 1$ we have that $T_t'(z)\geq T_t(y)/y$.
        \item $T_t$ exhibits super-linear growth in $[1,\,\infty)$. That is, $\frac{T_t(x)}{x}
        \geq
        \frac{T_t(y)}{y}
        \ , \text{ for }\;
        x\geq y\geq 1$.
        \item $T_t(2)\geq 3^{t}/2$ and for $x\geq 2$ we have that $T_t(x)\geq \frac{1}{4}\, x\, 3^{t}$.
    \end{enumerate}
\end{pro}

\begin{proof}
    We prove the items in order. The first two assertions of the first item can be shown using Eq. \eqref{eq T_n def} and for the third assertion one turns to Eq. \eqref{eq T_n forma 3}.

\pausa
  The first assertion of the second item follows from the Eq. \eqref{eq T_n recurrencia} and the previous item. Meanwhile, the second assertion follows from the Eq. \eqref{eq T_n derivada} and the previous item. 

\pausa
    To prove the third item we combine the previous ones in the following chain of inequalities:
    \[
    T_t'(z)
    \geq
    2\,t\, T_{t-1}(z)
    \geq
    2\,t\, T_{t-1}(y)
    =
    2\,y\,T_{t-1}(y) \frac{t}{y}
    \geq
    T_t(y)\frac{t}{y}
    \geq
    \frac{T_t(y)}{y}\,.
    \]
    
\pausa
 The next item is a consequence of the mean value theorem and the previous item. Indeed, given $x\geq y\geq 1$ the mean value theorem assures that there is a value $z\in[y,\,x]$ such that
    \[
    T_t(x)
    =
    T_t(y) + (x-y)T_t'(z)
    \geq
    T_t(y) + (x-y)\frac{T_t(y)}{y}
    =
    x\,\frac{T_t(y)}{y}
    \,.
    \]
    \noindent where the inequality in the middle is a consequence of the previous item.
        Finally, for the last item, we use Eq. \eqref{eq T_n forma 3} to see that
    \[
    T_t(2)
    =
    \frac{(2+\sqrt{3})^t+(2-\sqrt{3})^t}{2}
    \geq
    \frac{3^t}{2}
    \,,
    \]
    and then apply the fourth item to $x\geq 2\geq 1$ for the last estimation.
\end{proof}

\begin{pro}
    For $x\in[1,\,2]$ we have that $T_t(x)\geq 3^{t\,\sqrt{x-1}}/2$.   Therefore, we have that
    \[
    T_t(x)\geq \frac{1}{4}\ x\ 3^{t\cdot\min\{\sqrt{x-1},\,1\}}
    \ ,\text{ for }\;
    x\geq 1
    \,.
    \]
\end{pro}

\begin{proof}
    To prove the estimation in the interval $[1,\,2]$ We see that, for $x\in [1,\,2]$,
    \[
    2\,T_t(x)
    =
    (x+\sqrt{x^2-1})^t+(x-\sqrt{x^2-1})^t
    \geq
    (x+\sqrt{x^2-1})^t
    \]
    \noindent so it remains to prove that $(x+\sqrt{x^2-1})^t \geq 3^{t\sqrt{x-1}}=(3^{\sqrt{x-1}})^t$. Substituting $u=\sqrt{x-1}$, we need to prove that $ u^2+1+\sqrt{u^4+2u^2} - 3^u$ is non negative for $u\in[0,\,1]$. Dropping the term $u^4$, we instead will see that $f(u)=u^2+1+\sqrt{2}u-3^u$ is non negative in $[0,\,1]$. Since $f(0)=0$ it suffices to check that $f$ has non negative derivative. That is,
    \[
    f'(u)=2u+\sqrt{2}-\ln(3)\,3^u
    \geq 0
    \ , \text{ for }\;
    u\in[0,\,1]
    \,.
    \]

    \noindent Using the convexity of $3^u$ we can estimate 
    \[
    f'(u)
    \geq
    2u+\sqrt{2}-ln(3)\,(2u+1)
    =
    \sqrt{2}-\ln(3)-2\,(\ln(3)-1)\,u
    \geq
    \sqrt{2}-\ln(3)-2\,(\ln(3)-1)
    \geq 0
    \]

    \noindent so the proof for the estimation in the interval $[1,\,2]$ is complete. After that, the global estimation for the interval $[1,\,\infty)$ is achieved by combining the estimation we have just shown with the last item of the previous proposition.
\end{proof}

Our estimations for these polynomials improve slightly a result found in \cite{MuscoMusco} which is applied in the convergence analysis done there, and also in \cite{Drineas, WWZ}. For the sake of comparisons, we state our version of the result found in \cite{MuscoMusco}.

\begin{lem}
    Given  specified values $\alpha> 0$, $\gamma\in (0,\, 1]$,
and $t \geq 1$, there exists a degree $t$ polynomial $p(x)$ such that:
\begin{enumerate}
    \item $p((1+\gamma)\alpha)=(1+\gamma)\alpha$\,,
    \item $p(x)\geq x$ for all $x\geq(1+\gamma)\alpha$\,,
    \item $|p(x)|\leq\frac{(1+\gamma)2\alpha}{3^{t\sqrt{\gamma}}}$ for all $x\in[0,\,\alpha]$\,.
\end{enumerate}
Furthermore, when $t$ is odd, the polynomial only contains odd powered monomials.
\end{lem}

\begin{proof}
    The polynomial $p(x)=(1+\gamma)\alpha\frac{T_t(x/\alpha)}{T_t(1+\gamma)}$ satisfies all the needed conditions.
\end{proof}

\proof[Proof of item $2.$ in Proposition \ref{pro estimac poly main}]
We consider Notation \ref{notac2}, so that $A\in\C^{n\times n}$ is an Hermitian
matrix with eigenvalues $\la_1\geq \ldots\geq \la_n$. Let 
$T_t(x)$ be the Chebyshev polynomial of the first kind of degree $t$, and take the rescaled polynomials 
$\phi_t(x)=T_t((x-\la_n)/(\la_{p+1}-\la_n))$.
Using the previous facts about $T_t(x)$, $\phi_t$ satisfies the following:
\begin{enumerate}
    \item $|\phi_t(\la_j)|\leq 1$, for $p+1\leq j\leq n$. 
    \item $\phi_t(\la_j)\geq \frac{1}{4}\,\frac{\la_j-\la_n}{\la_{p+1}-\la_n}\ 3^{t\cdot \min\left \{\sqrt{\frac{\la_j-\la_n}{\la_{p+1}-\la_n}\,-\,1}\,,\,1\right\}}$,  for $1\leq j\leq p$.
\end{enumerate}
Then, we have that
\[
\|\phi_t(\Lambda_d)^{-1}\|_2 \; \|\phi_t(\Lambda_{p,\perp})\|_2 \ \leq \frac{1}{\phi_t(\lambda_d)} \leq 
4\, \frac{\la_{p+1}-\la_n}{\lambda_d-\la_n} \ 3^{-t\cdot\min\left\{\sqrt{\frac{\la_d-\la_n}{\la_{p+1}-\la_n} \,-\,1} \, ,\,1\right\}} \,.
\]
\QED

\end{document}